\documentclass[unnumsec,webpdf,modern,medium]{oup-authoring-template}

\onecolumn 

\usepackage[all,cmtip]{xy}
\usepackage[enableskew]{youngtab}

\newcommand{\C}{\mathbb{C}}
\newcommand{\R}{\mathbb{R}}

\newcommand{\N}{\mathbb{N}}
\newcommand{\flo}[1]{\left\lfloor\frac{#1}{2}\right\rfloor}
\newcommand{\Lor}{\overset{\circ}{L}}

\DeclareMathOperator{\Ann}{Ann}

\newcommand{\compmm}[1]{{\color{orange}#1}}

\theoremstyle{theorem}
\newtheorem{theorem}{Theorem}[section]

\newtheorem{lemma}[theorem]{Lemma}

\newtheorem{fact}[theorem]{Fact}

\theoremstyle{definition}

\newtheorem{remark}[theorem]{Remark}
\newtheorem{example}[theorem]{Example}

\begin{document}
\journaltitle{Journal Title Here}
\DOI{DOI HERE}
\copyrightyear{2025}
\pubyear{2025}
\access{Advance Access Publication Date: Day Month Year}
\appnotes{Paper}

	\title[Corrigendum]{Corrigendum to ``Higher Lorentzian polynomials, higher Hessians, and the Hodge-Riemann relations for graded oriented Artinian Gorenstein algebras in codimension two'', [International Mathematics Research Notices, Volume 2025, Issue 13, July 2025]}

\author[1]{Pedro Macias Marques}

\author[2,$\ast$]{Chris McDaniel}

\author[3]{Alexandra Seceleanu}	

\authormark{Macias Marques, McDaniel, Seceleanu}

\address[1]{\orgdiv{Departamento de Matem\'{a}tica, ECT, CIMA, IIFA}, \orgname{Universidade de \'{E}vora}, \orgaddress{\street{Rua Rom\~{a}o Ramalho, 59}, \postcode{P--7000--671 \'{E}vora}, \country{Portugal}},
pmm@uevora.pt}

\address[2]{\orgdiv{Department of Mathematics}, \orgname{Endicott College}, \orgaddress{\street{376 Hale St Beverly}, \postcode{MA 01915}, \country{USA}},
cmcdanie@endicott.edu}

\address[3]{\orgdiv{Department of Mathematics}, \orgname{University of Nebraska-Lincoln}, \orgaddress{\street{210 Avery Hall, 1144 T St, Lincoln}, \postcode{NE 68588}, \country{USA}}, aseceleanu@unl.edu
}

\corresp[$\ast$]{Corresponding author. \href{email:cmcdanie@endicott.edu}{cmcdanie@endicott.edu}}

%\email{pmm@uevora.pt}
	
%\email{cmcdanie@endicott.edu}
	
%\email{aseceleanu@unl.edu}

\abstract{A homogeneous bivariate $d$-form defines an $(i+1)$-rowed Toeplitz matrix for each $i$ between $0$ and $d$.  We use Hodge theory and Schur polynomials to prove that if the $(i+1)$-rowed Toeplitz matrix of a form is totally nonnegative, then so is the $i$-rowed one.  This fixes a gap in the main result of paper above.
	%We fix a gap in the proof of \cite[Theorem 2, Theorem 4.21]{MMS} and give other minor corrections from our paper \cite{MMS}.
	}     


\maketitle


\section{Introduction}
Shortly after its publication, we realized there was gap in our proof of Theorem 2 (or Theorem 4.21) in \cite{MMS}; in this corrigendum, we fix that gap and give a few other minor corrections.  We shall adhere to the same terminology as in \cite{MMS} for the most part, although to properly explain and repair our gap, we shall need the following notation. Fix positive integers $m,n$ satisfying $1\leq m\leq n$, and set $i=m-1$ and $n=d-i+1$.  We say that an $m\times n$ Toeplitz matrix $A=\phi^i_d(F)\in\mathcal{T}(m,n)$
%, we define its \emph{Toeplitz derivative} to be the $(m-1)\times (n+1)$ Toeplitz matrix $\partial A=\phi^{i-1}_d(F)\in\mathcal{T}(m-1,n+1)$.  Then we say that $A$ 
is \emph{strongly totally positive}, respectively \emph{strongly totally nonnegative}, if, for each $j$ with $0\leq j\leq i$,  the Toeplitz matrix $\phi^j_d(F)$ is totally positive, respectively totally nonnegative.  %In other words, $\phi^i_d(F)$ is strongly totally positive or nonnegative if $\phi^j_d(F)$ is totally positive or nonnegative for all $0\leq j\leq i$.  
It is relatively straightforward to see that strongly totally positive and totally positive are actually equivalent, so that in the positive case, the adjective strongly is redundant (Lemma \ref{lem:super}).  The gap in our original proof of \cite[Theorem 4.21]{MMS}, is that we had assumed that the same thing is true in the nonnegative case; it is, as it turns out, but this is far from obvious.  The proof given here relies on the theory of Schur polynomials, and in particular the Littlewood-Richardson rule.
In this corrigendum, we prove the following results, which together, constitute a complete proof of \cite[Theorem 4.21]{MMS}.
\begin{theorem}
	\label{thm:Lorentzian2}
	Let $F\in Q_d$ be a homogeneous $d$-form, and fix $0\leq i\leq \flo{d}$.	The following are equivalent.
	\begin{enumerate}
		\item $F$ is $i$-Lorentzian
		\item $\phi^i_d(F)$ is strongly totally nonnegative.
		\item $A_F$ satisfies mixed HRR$_i$ on the standard open convex cone of linear forms
		$$U=\left\{ax+by \ | \ (a,b)\in\R^2_{>0}\right\}.$$
	\end{enumerate}
\end{theorem} 	

\begin{theorem}
	\label{thm:STN=TN}
	With notation as in Theorem \ref{thm:Lorentzian2}:
	If $\phi^i_d(F)$ is totally nonnegative, then $A_F$ satisfies the mixed HRR$_i$ on $U$.
\end{theorem}
Note that Theorem \ref{thm:Lorentzian2} together with Theorem \ref{thm:STN=TN} imply that strongly totally nonnegative and totally nonnegative are equivalent, and hence we recover \cite[Theorem 4.21]{MMS}.  
%It reminds one of the 17 camels trick \cite{MO}: in the end, we did not really need the strong version of total nonnegativity at all, but having it there seems to make our arguments work out nicely.  
This corrigendum is organized as follows.  In Section \ref{sec:Prelim} we establish some fundamental properties of Toeplitz matrices and prove that totally positive Toeplitz matrices are always strongly totally positive.  In Section \ref{sec:Schur} we review the prerequisite theory of Schur polynomials and their relation to Toeplitz minors.  In Section \ref{sec:MixedHessians}, we give a Pl\"ucker expansion formula for the mixed Hessian determinant, and identify a certain specialization of it.  In Section \ref{sec:Theorem1} we prove Theorem \ref{thm:Lorentzian2}, and in Section \ref{sec:Theorem2} we prove Theorem \ref{thm:STN=TN}.  In Section \ref{sec:Other}, we give other minor, mostly typographical, corrections to our paper \cite{MMS}.

Notation:  We shall adhere to the notational conventions of \cite{MMS} for the most part, but we shall adopt the following notation here:  Since most of our indexing will start with $0$, we shall use the notation $[n]_0=\{0,1,\ldots,n\}$, and occasionally we shall use the notation
$\binom{[n]_0}{k}$ for the set of $k$-subsets of the set $[n]_0$.  For an $(m+1)\times (n+1)$ matrix $A$, a fixed integer $k$ satisfying $1\leq k\leq \min\{m,n\}+1$, and $k$-subsets $I\in\binom{[m]_0}{k}$ and $J\in\binom{[n]_0}{k}$, 
we shall write $A_{IJ}$ for the submatrix of $A$ whose rows and columns are indexed by \(I\) and \(J\), respectively, and write $\Delta_{IJ}(A)$ for the corresponding minor determinant $\det(A_{IJ})$.
%we shall write $A_{IJ}$ for the corresponding submatrix $A$ and write $\Delta_{IJ}(A)$ for the corresponding minor determinant $\det(A_{IJ})$.  
We say that the submatrix or the minor is consecutive if both $I$ and $J$ consist of consecutive integers, and we say the submatrix or minor is initial if it is consecutive and either $0\in I$ or $0\in J$.    

As in \cite{MMS}, $\mathcal{M}(m,n)$ denotes the Euclidean space of $m\times n$ matrices, $Q=\R[X,Y]$ 
the standard graded polynomial ring in two variables,
%the polynomial ring of homogeneous $d$-forms, 
$R=\R[x,y]$ the polynomial ring of partial differential operators on $Q$, i.e. $x\circ F=\frac{\partial F}{\partial X}$ and $y\circ F=\frac{\partial F}{\partial Y}$, and for $F\in Q_d$ a homogeneous $d$-form $A_F=R/\Ann(F)$ the corresponding graded oriented Artinian Gorenstein algebra.  
For $F\in Q_d$, $s(F)=s(A_F)$, the Sperner number of $F$ is the number $s(F) = \max\{\dim_\R A_i \mid  0 \leq i \leq d \}$. For $0\leq i\leq s-1$, $\operatorname{MHess}_i(F,\mathcal{E})|_{(\underline{X},\underline{Y})}$ is the $i^{th}$ mixed Hessian matrix of $F$ with respect to the fixed monomial basis $\mathcal{E}=\left\{e^i_p=x^py^{i-p} \ | \ 0\leq p\leq i\right\}$, whose entries are polynomials of degree $d-2i$ in the $2(d-2i)$-variables $(\underline{X},\underline{Y})=(X_1,\ldots,X_{d-2i},Y_1,\ldots,Y_{d-2i})$.
The standard open cone is $U=\left\{ax+by \ | \ a,b>0\right\}\subset R_1$ and the standard closed cone is 
\({\overline{U}=\left\{ax+by \ | \ a,b\geq 0, \ (a,b) \ne (0,0)\right\}}\); 
%$\overline{U}=\left\{ax+by \ | \ a,b\geq 0, \ (a,b)=(0,0)\right\}$; 
we shall abuse notation slightly and also refer to the standard open and closed cones in the quotient algebra $A_F$ by the same notation, when it is clear from the context what we mean.  For integers $i,d$ satisfying $0\leq i\leq \flo{d}$, and $F\in Q_d$, we write $F=\sum_{k=0}^d\binom{d}{k}c_kX^kY^{d-k}$, 
we call $(c_0,\ldots,c_d)\in\R^{d+1}$ normalized coefficient sequence of \(F\), 
%we call $(c_0,\ldots,c_d)\in\R^{d+1}$ its normalized coefficient sequence, 
and define its $i^{th}$ Toeplitz matrix by 
\[
\phi^i_d(F)=\begin{pmatrix}
c_i & \cdots & c_d\\
\vdots & \ddots & \vdots\\
c_0 & \cdots & c_{d-i} \\ 
\end{pmatrix}
=\left(c_{i+q-p}\right)_{\substack{0\leq p\leq i\\ 0\leq q\leq d-i\\}}.
\]
%\[
%\phi^i_d(F)=\left(\begin{array}{cccc} c_i & c_{i+1} & \cdots & c_d\\
%	\vdots & \ddots & \ddots & \ddots\\
%	c_0 & \ddots & \ddots & \ddots\\ \end{array}\right)=\left(c_{i+q-p}\right)_{\substack{0\leq p\leq i\\ 0\leq q\leq d-i\\}}.
%\]

\section{Toeplitz Fundamentals}
\label{sec:Prelim}
%Recall that given a matrix $A\in\mathcal{M}(m,n)$, a positive integer $k$, $1\leq k\leq\min\{m,n\}$, two $k$-subsets $I\in\binom{[m]_0}{k}$ and $J\in\binom{[n]_0}{k}$, the $k\times k$ submatrix $A_{IJ}$ is obtained from $A$ using the rows indexed by $I$ and the columns indexed by $J$.  We say that the submatrix is \emph{consecutive of size $k$} if both $I$ and $J$ consist of consecutive indices, and we say that it is \emph{initial (of size $k$)} if it is consecutive and at least one of $I$ or $J$ contains the first index of $[n]$ or $[m]$ (our indexing may start at $0$ or $1$, depending on the context).  A (consecutive or initial) minor is the determinant of a (consecutive or initial) submatrix.  
The following result expounds upon \cite[Remark 4.17]{MMS}.
\begin{lemma}
	\label{lem:Properties}
	Let $F\in Q_d$ be any homogeneous $d$-form and fix $i$, with $1\leq i\leq \flo{d}$.
	\begin{enumerate}
		\item Every consecutive submatrix of the Toeplitz matrix $\phi^i_d(F)$ is also an initial submatrix.
		
		\item The set of distinct consecutive submatrices of $\phi^{i-1}_d(F)$ is equal to the set of distinct consecutive submatrices of $\phi^{i}_d(F)$ of size $\leq i$. 
	\end{enumerate}
\end{lemma}
\begin{proof}
	For (1), fix $k$, with $1\leq k\leq i+1$, and fix consecutive $k$-subsets
	$$\begin{array}{ll}
		A= & \left\{0\leq a_0, a_0+1,\ldots,a_0+(k-1)\leq i\right\}\\
		B= & \left\{0\leq b_0,b_0+1,\ldots,b_0+(k-1)\leq d-i+1\right\}\\
		\end{array}$$
	If either $a_0=0$ or $b_0=0$, they define an initial submatrix.  Otherwise, we may assume that $a_0,b_0>0$.  Let $t=\min\{a_0,b_0\}>0$, and define the shifted sets 
	$$\begin{array}{lll}
		A'= & A-t= & \left\{0\leq a_0-t, a_0-t+1,\ldots,a_0-t+(k-1)\leq i\right\}\\
		B'= & B-t= & \left\{0\leq b_0-t,b_0-t+1,\ldots,b_0-t+(k-1)\leq d-i+1\right\}\\
	\end{array}$$
	(note either $A$ or $B$ must contain an initial index).
	Then we have 
	$$\phi^i_d(F)_{AB}=\left(c_{(b_0+s)-(a_0+r)+i}\right)_{0\leq r,s\leq k-1}=\left(c_{(b_0-t+s)-(a_0-t+r)+i}\right)_{0\leq r,s\leq k-1}=\phi^i_d(F)_{A'B'}$$
	and since $\phi^i_d(F)_{A'B'}$ is initial, (1) is proved.
	
	For (2), we first show that every consecutive submatrix of $\phi^{i-1}_d(F)$ is a consecutive submatrix of $\phi^i_d(F)$.  To this end, fix $k$, $1\leq k\leq i$, and fix two consecutive $k$-subsets $A\subset\{0,\ldots,i-1\}$ and $B\subseteq\{0,\ldots,d-i+1\}$; by (1), we may assume one of them, say $A$, is initial.  Write
	$$\begin{array}{ll}
		A= & \left\{0,1,\ldots,k-1\right\}\subseteq \left\{0,\ldots,i-1\right\}\\
		B= & \left\{0\leq b_0,b_0+1,\ldots,b_0+(k-1)\leq d-i+1\right\}\\
	\end{array}$$
	There are two cases to consider:
	
	\textbf{Case 1:}  $d-i+1\notin B$.  In this case $B\subset\{0,\ldots,d-i\}$ and since $A\subset\{0,\ldots,i-1\}$, $A'=A+1\subset\{0,\ldots,i\}$, and we have
	$$\phi^{i-1}_d(F)_{AB}=\left(c_{(b_0+s)-(r)+i-1}\right)_{0\leq r,s\leq k-1}=\left(c_{(b_0+s)-(r+1)+i}\right)_{0\leq r,s\leq k-1}=\phi^{i}_d(F)_{A'B}.$$
	
	\textbf{Case 2:} $d-i+1\in B$.  In this case, $B'=B-1\subset\{0,\ldots,d-i\}$ and, of course $A\subset\{0,\ldots,i-1,i\}$, hence
	\[
	\phi^{i-1}_d(F)_{AB}=\left(c_{(b_0+s)-(r)+i-1}\right)_{0\leq r,s\leq k-1}=\left(c_{(b_0+s-1)-(r)+i}\right)_{0\leq r,s\leq k-1}=\phi^{i}_d(F)_{AB'}.
	\]
	
	For the converse, we want to show that every consecutive submatrix of $\phi^i_d(F)$ of size $\leq i$ is also a consecutive submatrix of $\phi^{i-1}_d(F)$.  As above, fix $k$, with $1\leq k\leq i$, and fix $k$-subsets $A\subset\{0,\ldots,i\}$ and $B\subset\{0,\ldots,d-i\}$ with $A$ initial.  Then $A\subset\{0,\ldots,i-1\}$ and $B\subset\{0,\ldots,d-i\}$, and thus $B'=B+1=\{0,\ldots,d-i+1\}$, and we have 
	$$\phi^i_d(F)_{AB}=\left(c_{(b_0+s)-(r)+i}\right)_{0\leq r,s\leq k-1}=\left(c_{(b_0+s+1)-(r)+i-1}\right)_{1\leq r,s\leq k-1}=\phi^{i-1}_d(F)_{AB'},$$
	as desired.  
\end{proof} 
Intuitively, we can imagine the two Toeplitz matrices, superimposed and justified by their lower left corners, then any square submatrix of the shorter $\phi^{i-1}_d(F)$ is either already a submatrix of the taller $\phi^i_d(F)$ or else it can be translated along the constant diagonals to an identical submatrix of the taller one; see Figure \ref{fig:Tslide}.

\begin{figure}
	\begin{center}
	\begin{tikzpicture}[scale=1]
		\draw (0,0) rectangle (11,3);
		\draw (0,0) rectangle (10,4);
		\draw (5,1.8) rectangle (7.2,4);
		\draw (8.5,.3) rectangle (10.7,2.5);
		\draw[dotted] (5,1.8) -- (8.5,.3);
		\draw[dotted] (7.2,4) -- (10.7,2.5);
		\draw[dotted] (7.2,1.8) -- (10.7,.3);
		\draw[dotted] (5,4) -- (8.5,2.5);
		\node[above right] at (9,4) {$\phi_d^{i}(F)$};
		\node[right] at (11,3) {$\phi_d^{i-1}(F)$};
	\end{tikzpicture}
	\end{center}
	\caption{Translating a consecutive submatrix of $\phi^{i-1}_d(F)$ into an initial submatrix of $\phi^i_d(F)$}
	\label{fig:Tslide}
\end{figure}
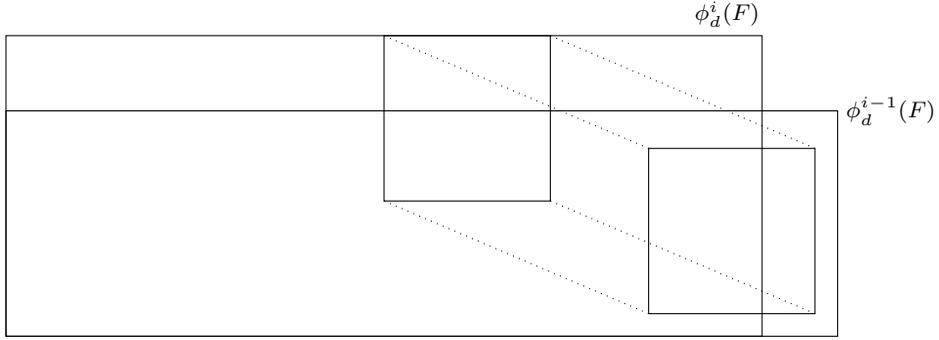  
%Lemma \ref{lem:Properties}(1) was observed in \cite[Remark 4.7]{MMS}, but (2) was only haphazardly mentioned and then misused in our faulty proof of \cite[Theorem 4.21]{MMS}. 
Using Lemma \ref{lem:Properties}, we show that the notion of strongly totally positive is redundant.
\begin{lemma}
	\label{lem:super}
	Let $F\in Q_d$ be any homogeneous $d$-form.  The Toeplitz matrix $\phi^i_d(F)$ is totally positive if and only if $\phi^j_d(F)$ is also totally positive for all $0\leq j\leq i$.
\end{lemma}
\begin{proof}
	Applying Lemma \ref{lem:Properties} inductively, we see that the consecutive minors of $\phi^i_d(F)$ of size $\leq j+1$ are exactly the consecutive minors of $\phi^j_d(F)$, for each $0\leq j\leq i$.  In particular, if $\phi^i_d(F)$ is totally positive, then, for each $0\leq j\leq i$ all consecutive minors of $\phi^j_d(F)$ must be positive.  By Fekete's theorem (\cite[Fact 4.6]{MMS} or \cite[Corollary 3.1.5]{FJ}), which states that a matrix whose consecutive minors are all positive must be totally positive, it follows that $\phi^j_d(F)$ is totally positive.  Conversely if $\phi^j_d(F)$ is totally positive for all $0\leq j\leq i$, it follows that every consecutive minor of $\phi^i_d(F)$ must be positive and hence, again by Fekete's theorem, $\phi^i_d(F)$ must be totally positive too.   
\end{proof}

%Our ``mistake'' in \cite{MMS} was assuming that Lemma \ref{lem:super} also holds for totally nonnegative Toeplitz matrices; this remains an open problem.  Concretely, if $\phi^i_d(F)$ is totally nonnegative, does it necessarily follow that $\phi^{i-1}_d(F)$ must also be totally nonnegative?  
%it does not, as the following example shows.
%\begin{example}
%	\label{ex:Not}
%	Let $F=X^4+Y^4$.  Then the Toeplitz matrix 
%	$$\phi^2_4(F)=\left(\begin{array}{ccc}
%		0 & 0 & 1\\ 0 & 0 & 0\\ 1 & 0 & 0\\ \end{array}\right)$$
%	is totally nonnegative, but not strongly totally nonnegative because its first Toeplitz derivate, namely 
%	$$\phi^1_4(F)=\left(\begin{array}{cccc}
%		0 & 0 & 0 & 1\\ 1 & 0 & 0 & 0\\ \end{array}\right)$$
%	is not totally nonnegative.
%	One can further see that the totally nonnegative Toeplitz matrix $\phi^2_4(F)$ cannot possibly be the limit of a sequence of totally positive Toeplitz matrices as follows.  Assume to the contrary that there is some sequence of totally positive Toeplitz matrices converging to $\phi^2_4(F)$, say $\phi^2_4(F_n)\to \phi^2_4(F)$ (according to \cite[Proposition 4.8]{MMS}, the $F_n$ would be strictly $2$-Lorentzian polynomials).  On the other hand, according to Lemma \ref{lem:super}, the Toeplitz derivates $\phi^1_4(F_n)$, must be totally positive as well, and hence their limit $\phi^1_4(F_n)\to\phi^1_4(F)$ must be totally nonnegative, but it is not!
%\end{example}

\begin{remark}
	\label{rem:STNN}
	The result \cite[Theorem 2, Theorem 4.21]{MMS} implies that every totally nonnegative Toeplitz matrix, say $\phi^i_d(F)$, must be the limit of a sequence of totally positive Toeplitz matrices, say $\phi^i_d(F_n)\rightarrow \phi^i_d(F)$.  Since $\phi^i_d(F_n)$ are all strongly totally positive, by Lemma \ref{lem:super}, it follows then that the limiting Toeplitz matrix $\phi^i_d(F)$ must be strongly totally nonnegative.  In particular, we seem to have deduced the fact that totally nonnegative Toeplitz are always strongly totally nonnegative.  The problem here is that the proof of the result \cite[Theorem 2, Theorem 4.21]{MMS} given there \emph{uses} that fact \emph{a priori}, and thus creates an unfortunate circle in this argument.  
\end{remark}
There is one more property of totally positive Toeplitz matrices that we need.  Recall that a matrix $A\in\mathcal{M}(m,n)$ is called \emph{totally positive of order $k$}, or TP$_k$, 
if all of its $j\times j$ minors are positive for all \(j\), with $1\leq j\leq k$. 
%if all of its $j\times j$ minor determinants are positive for all $1\leq j\leq k$. 
	\begin{lemma}
		\label{lem:TPs}
		Let $F\in Q_d$ be any homogeneous $d$-form with Sperner number $s=s(F)$.
		The Toeplitz matrix $\phi^{s-1}_d(F)$ is totally positive if and only if $\phi^j_d(F)$ is TP$_s$ and totally nonnegative for all 
		\(j\), with $s\leq j\leq \flo{d}$. 
		%$s\leq j\leq \flo{d}$. 
	\end{lemma}
	\begin{proof}
	We appeal to \cite[Corollary 3.1.7]{FJ}, which implies $A\in\mathcal{M}(m,n)$ is TP$_k$, $1\leq k\leq \min\{m,n\}$ if and only if every consecutive minor of $A$ of size $\leq k$ is positive.  
	
	If $\phi^{s-1}_d(F)$ is totally positive and $j$ satisfies $s\leq j\leq \flo{d}$, then inductively, by Lemma \ref{lem:Properties}, every consecutive minor of $\phi^j_d(F)$ of size $\leq s$ must be positive, and hence, by the above cited result, $\phi^j_d(F)$ must be TP$_s$.  Moreover, since $\operatorname{rank}(\phi^j_d(F))=\min\{j+1,s\}$ (\cite[Lemma 4.5]{MMS}), it follows that every minor of $\phi^j_d(F)$ of size $\geq s+1$ is zero, and hence $\phi^j_d(F)$ must be totally nonnegative.
	
	Conversely suppose that $\phi^j_d(F)$ is TP$_s$ and totally nonnegative for all $j$, $s\leq j\leq \flo{d}$.  Then all consecutive minors of $\phi^j_d(F)$ of size $\leq s$ must be positive, and hence by Lemma \ref{lem:Properties}, all consecutive minors of $\phi^{s-1}_d(F)$ must also be positive, which implies, by Fekete's theorem \cite[Fact 4.6]{MMS}, that $\phi^{s-1}_d(F)$ must be totally positive.
\end{proof}

\section{Schur polynomials and totally nonnegative Toeplitz matrices}
\label{sec:Schur}
\subsection{A Primer on Schur Polynomials}
Our main reference for this section is Fulton's book \cite{F}; other references are \cite{LT}, \cite{Smith} and \cite{Stanley}.
For any integer partition $\lambda=(\lambda_0,\ldots,\lambda_r)$, define its Young diagram to be the left justified weakly decreasing array of boxes with $\lambda_i$ boxes in row $i$ (note we start our indexing at $i=0$), and a tableau $T$ of shape $\lambda$ is any filling of those boxes with the numbers $[n-1]_0=\{0,\ldots,n-1\}$.  We shall say that $T$ is \emph{semi-standard} if the entries of $T$ are strictly increasing down the columns and weakly increasing left to right along the rows.  The set of semi-standard Young tableaux of shape $\lambda$ is denoted by $\operatorname{SSYT}(\lambda)$, and for each $T\in\operatorname{SSYT}(\lambda)$, define its $T$-monomial 
$$\mathbf{x}^T=\prod_{i=1}^nx_i^{T(i-1)}$$
where $T(i-1)$ is the number of times $i-1$ appears in $T$.  Then the Schur polynomial corresponding to the partition $\lambda$ is the sum of all $T$-monomials where $T$ ranges over the set of semi-standard Young tableaux, i.e. 
$$s_\lambda=s_\lambda(x_1,\ldots,x_n)=\sum_{T\in\operatorname{SSYT}(\lambda)}\mathbf{x}^T.$$
A fundamental fact about Schur polynomials is that that they form a vector space basis for the subring of symmetric functions $\C[x_1,\ldots,x_n]^{\mathfrak{S}_n}$ in the polynomial ring $\C[x_1,\ldots,x_n]$; for a proof, see \cite[Proposition 6.1]{F}.

Next, given another partition $\mu=(\mu_0,\ldots,\mu_r)$ satisfying $\mu_p\leq \lambda_p$ for all $0\leq p\leq r$, define the skew Young diagram of shape $\lambda/ \mu$ by removing the boxes in the Young diagram of $\mu$ from the Young diagram of $\lambda$.  Then a skew tableau, a semi-standard skew tableau, and the skew Schur polynomials are defined analogously.  On the other hand, since the (non-skew) Schur polynomials form a basis for the symmetric functions, and since skew Schur polynomials are also symmetric, it follows that we have 
\begin{equation}
	\label{eq:LR}
	s_{\lambda/\mu}=\sum_{\nu}c^{\lambda}_{\mu\nu}\cdot s_{\nu}
\end{equation}   
for some coefficients $c^{\lambda}_{\mu\nu}\in\R$.  It turns out that the numbers $c^{\lambda}_{\mu\nu}$ are actually nonnegative integers, called the \emph{Littlewood-Richardson coefficients} for the triple $(\lambda,\mu,\nu)$.  To characterize them combinatorially, we need some more terminology.  For any tableau $T\in\operatorname{SSYT}(\lambda/\mu)$ define its row word $w_{\text{row}}(T)$ to be the word obtained by writing the entries of $T$ in a single row, starting from the bottom row, reading from left to right, then bottom to top.  A word $w=w_1,w_2,\ldots,w_s$ is called a \emph{reverse lattice word} if when read backwards to any letter, say $w_s,w_{s-1},\ldots,w_r$, the resulting sequence contains at least as many $0$'s as $1$'s, at least as many $1$'s as $2$'s and so on.  A tableau $T\in\operatorname{SSYT}(\lambda/\mu)$ is called  a \emph{Littlewood-Richardson tableau of shape $\lambda/\mu$} if besides being semi-standard, its row word is a reverse lattice word; denote the set of Littlewood-Richardson tableaux $\operatorname{LRT}(\lambda/\mu)$.  Note that the content of $T\in\operatorname{LRT}(\lambda/\mu)$, i.e. the integer sequence $\nu(T)=(\nu_0,\ldots,\nu_m)$ where $\nu_i$ is the number of $i$'s in $T$, is a partition; in this case we call $\nu(T)$ a \emph{Littlewood-Richardson partition} for the skew shape $\lambda/\mu$, and we denote the set of all Littlewood-Richardson partitions by $\operatorname{LRP}(\lambda/\mu)$, and let $\operatorname{LRT}_\nu\left(\lambda/\mu\right)$ denote the set of Littlewood-Richardson tableaux of shape $\lambda/\mu$ and content $\nu$.  The following is \cite[Proposition 3, page 64]{F}, and we refer the reader there for a proof.
\begin{fact}[Littlewood-Richardson rule]
	\label{fact:LRRule}
	The Littlewood-Richardson coefficient $c^{\lambda}_{\mu,\nu}$ equals the number of Littlewood-Richardson tableaux of shape $\lambda/\mu$ and of content $\nu$, i.e.
	$$c^{\lambda}_{\mu,\nu}=\# \operatorname{LRT}_{\nu}(\lambda/\mu).$$  
	In particular, the Littlewood-Richardson coefficient $c^{\lambda}_{\mu,\nu}$ is nonzero if and only if $\operatorname{LRT}_\nu(\lambda/\mu)\neq \emptyset$ if and only if $\nu\in\operatorname{LRP}(\lambda/\mu)$.
\end{fact}

Note that by Equation \ref{eq:LR} and Fact \ref{fact:LRRule}, it follows that if, for some fixed $(a_1,\ldots,a_n)\in\C^n$, $s_\nu(a_1,\ldots,a_n)\geq 0$ for all partitions $\nu\in \operatorname{LRP}(\lambda/\mu)$, then $s_{\lambda/\mu}(a_1,\ldots,a_n)\geq 0$ too.

For any fixed positive integers $n,i$, denote by $h_i=h_i(z_1,\ldots,z_n)$ the $i^{th}$ complete symmetric polynomial in $n$ variables.  One version of the fundamental theorem of invariant theory, c.f. \cite[Proposition 1, page 73]{F} or \cite[Corollary 7.6.2]{Stanley}, states that the subring of symmetric functions $\C[x_1,\ldots,x_n]^{\mathfrak{S}_n}$ is generated as an algebra by the complete symmetric polynomials, i.e.
$$\C[x_1,\ldots,x_n]^{\mathfrak{S}_n}=\C[h_1,\ldots,h_n]$$ 
Since the skew Schur functions are symmetric, it follows that they should be expressible as polynomials in the complete symmetric polynomials.  This is the (generalized) Jacobi-Trudi identity, c.f. \cite[Exercise 7, page 77]{F} or \cite[Theorem 7.16.1]{Stanley}.
\begin{fact}[Generalized Jacobi-Trudi Identity]
	\label{fact:GJT}
	For any partitions $\mu=\left(\mu_0,\ldots,\mu_r\right)\subset\lambda=\left(\lambda_0,\ldots,\lambda_r\right)$, we have 
	$$s_{\lambda/\mu}=\det\left(\left(h_{\lambda_p-\mu_q+p-q}\right)_{0\leq p,q\leq r}\right).$$
\end{fact}

We will need the following result, which is \cite[Theorem 3.15]{LT}; it also appears as a remark in \cite[page 122]{Smith}.  We omit the proof for the sake of brevity.
\begin{fact}
	\label{lem:HSurj}
	With $h_1,\ldots,h_n\in\C[x_1,\ldots,x_n]$ the complete symmetric polynomials as above, the polynomial map 
	\begin{equation}
		\label{eq:Surj}
		H\colon \C^n\rightarrow\C^n, \ H(x_1,\ldots,x_n)=\left(h_1(x_1,\ldots,x_n),\ldots,h_n(x_1,\ldots,x_n)\right)
	\end{equation}
	is surjective.
\end{fact}

\subsection{From Toeplitz minors to Schur polynomials}
Fix $F\in Q_d$, with normalized coefficient sequence $(c_0,\ldots,c_d)\in\R^{d+1}$, and let $a=a(F)$ be the smallest index $0\leq a\leq d$, for which $c_a\neq 0$.  Then, by Lemma \ref{lem:HSurj}, there exists complex numbers $(z_1,\ldots,z_{d-a})\in\C^{d-a}$ for which 
$$h_{b-a}(z_1,\ldots,z_{d-a})=\frac{c_b}{c_a}, \ \forall a\leq b\leq d;$$
note that we can extend this rule to all $b\leq d$ since $h_{b-a}\equiv 0$ for $b-a<0$.  It follows then that for any fixed $r$, with $0\leq a\leq  r\leq \flo{d}$ and for any $r+1$-column indexing subset 
$$K=\left\{0\leq k_0<\cdots<k_r\leq d-r\right\}$$
the maximal minor of $\phi^r_d(F)$ corresponding to $K$ satisfies
\begin{equation}
	\label{eq:DeltaK}
	\Delta_K(\phi^r_d(F))=\det\left(\left(c_{r+k_q-p}\right)_{0\leq p,q\leq r}\right)=c_a^{r+1}\cdot \det\left(\left(h_{\nu_{p}-p+q}\right)_{0\leq p,q\leq r}\right)=c_a^{r+1}\cdot s_{\nu}
\end{equation}
by Fact \ref{fact:GJT}, where $\nu=(\nu_0,\ldots,\nu_r)$ is the partition defined from $K$ by 
\begin{equation}
	\label{eq:nu}
	\nu_{r-q}=(k_q-q)+(r-a), \ \ 0\leq q\leq r.
\end{equation}
To see the second equality in Equation \eqref{eq:DeltaK}, we should apply the determinant-preserving transformation
$$\left(a_{pq}\right)_{0\leq p,q\leq r}\mapsto \left(a_{(r-q)(r-p)}\right)_{0\leq p,q\leq r}.$$

More generally, for any fixed $i\geq r\geq a\geq 0$, and any fixed row and column indexing sets for $\phi^i_d(F)$,
\begin{align*}
	I = & \left\{0\leq i_0<\cdots<i_r\leq i\right\}\\
	J = & \left\{0\leq j_0<\cdots<j_r\leq d-i\right\}
\end{align*}
define the partitions $\lambda=(\lambda_0,\ldots,\lambda_r)$ and $\mu=(\mu_0,\ldots,\mu_r)$ by 
\begin{align}
	\label{eq:lambda0}
	\lambda_p= & (j_r-r)+i-(i_p-p), \ 0\leq p\leq r\\
	\label{eq:mu0}
	\mu_q = & (j_r-r)-(j_q-q), \ \ 0\leq q\leq r.
\end{align}
Then $\lambda_p-(\mu_p+a)=(i-a)-(i_p-p)+(j_q-q)\geq 0$, and then the corresponding $(r+1)\times (r+1)$ minor of $\phi^i_d(F)$ is the skew Schur function
$$\Delta_{IJ}\left(\phi^i_d(F)\right)=\det\left(\left(c_{i+j_q-i_p}\right)_{0\leq p,q\leq r}\right)=c_a^{r+1}\cdot\det\left(\left(h_{\lambda_p-\mu_q+q-p-a}\right)_{0\leq p,q\leq r}\right)=c_a^{r+1}\cdot s_{\lambda/(\mu+a)},$$
where $\mu+a$ is the partition $\mu+a=(\mu_0+a,\ldots,\mu_r+a)$.

\begin{lemma}
	\label{lem:LRr}
	Fix integers $a,d,i,r$ satisfying $0\leq a<r\leq i\leq \flo{d}$ and assume that $I=\{0\leq i_0<\cdots<i_r\leq i\}$ and $J=\{0\leq j_0<\cdots<j_r\leq d-i\}$ are given with $\lambda$ and $\mu$ defined as in Equation \eqref{eq:lambda0} and Equation \eqref{eq:mu0}.  Then every Littlewood-Richardson partition $\nu=(\nu_0,\ldots,\nu_m)\in\operatorname{LRP}(\lambda/(\mu+a))$ satisfies
	$m=r$ and $\nu_0\leq d-r-a$ and $\nu_r\geq r-a$.  In particular, every $(r+1)\times (r+1)$ minor of $\phi^i_d(F)$ can be written as a nonnegative linear combination of the maximal minors of $\phi^r_d(F)$; explicitly we have,
	\begin{equation}
		\label{eq:DeltaLR}
		\Delta_{IJ}\left(\phi^i_d(F)\right)=\sum_{\nu\in\operatorname{LRP}(\lambda/(\mu+a))}c^{\lambda}_{(\mu+a),\nu}\cdot\Delta_{K(\nu)}\left(\phi^r_d(F)\right)
	\end{equation}
	where $K(\nu)=K=\{0\leq k_0<\cdots<k_r\leq d-r\}$ is defined as in Equation \eqref{eq:nu}.	
\end{lemma}
\begin{proof}
	Fix a Littlewood-Richardson partition $\nu=(\nu_0,\ldots,\nu_m)\in\operatorname{LRP}(\lambda/(\mu+a))$, and let $T\in\operatorname{LRT}_\nu\left(\lambda/(\mu+a)\right)$ be any Littlewood-Richardson tableau of shape $\lambda/(\mu+a)$ and content $\nu$.
	Note that $$\lambda_p-(\mu_0+a)=(i-a)-(i_p-p)+j_0\geq r-a, \ \forall 0\leq p\leq r.$$
	It follows that the Young diagram for the skew shape $\lambda/(\mu+a)$ must have at least $r-a$ columns with exactly $r+1$ boxes.    
	By the reverse lattice word property, the $i^{th}$ row can only contain entries $\leq i$ for all $0\leq i\leq r$.  It follows that only the numbers  from the set $\{0,\ldots,r\}$ can occur in $T$, and moreover by the column-strict condition every number from that set must be used at least once.  Therefore it follows that $m=r$, and since the last row of $T$ must contain at least $(r-a)$ $r$'s, $\nu_r\geq r-a$ as well. 
	
	Finally, since $0$'s can only be placed in the top row of any given column, and there are $\lambda_0-a=(j_r-r)+i-i_0-a\leq d-r-a$ columns in the skew shape $\lambda/(\mu+a)$, it follows that $\nu_0\leq d-r-a$.
	
	For the last statement, note that every partition $\nu'=(\nu'_0,\ldots,\nu'_r)$ satisfying $\nu_0'\leq d-2r$ corresponds uniquely to a maximal minor of $\Delta_K\left(\phi^r_d(F)\right)$ according to the rule $\nu'_{r-q}=k_q-q$.  Hence given any Littlewood-Richardson partition $\nu=(\nu_0,\ldots,\nu_r)\in\operatorname{LRP}(\lambda/(\mu+a))$, define $\nu'=(\nu'_0,\ldots,\nu'_r)$, $\nu'_i=\nu_i-(r-a)$, and the second statement follows from our computations above and Equation \eqref{eq:LR}. 
\end{proof}

\begin{remark}
	\label{rem:BD}
	It seems that a similar formula to Equation \eqref{eq:DeltaLR} was also discovered by Bump-Diaconis \cite[Lemma page 261]{BD}.  %In fact, our formula may be interpreted as a refinement of theirs in the sense that we replace their ``$\lambda+N^n$ for sufficiently large $N$'' with an explicit choice for $\lambda$ given by Equation \eqref{eq:lambda0}.  
	Our Formulas \eqref{eq:lambda0} and \eqref{eq:mu0} were borrowed from Sagan \cite[Theorem 6.3]{Sagan}.
\end{remark}

\begin{example}
	Take $d=9$, $a=1$ $r=2$ and $i=3$, so that 
	$$\phi^3_9(F)=\left(\begin{array}{ccccccc}
		c_3 & c_4 & c_5 & c_6 & c_7 & c_8 & c_9\\
		c_2 & c_3 & c_4 & c_5 & c_6 & c_7 & c_8\\
		c_1 & c_2 & c_3 & c_4 & c_5 & c_6 & c_7\\
		c_0 & c_1 & c_2 & c_3 & c_4 & c_5 & c_6\\
	\end{array}\right)$$
	and $c_0=0$ and $c_1\neq 0$.
	Define the subsets 
	\begin{align*}
		I= & \left\{0, 1, 3\right\}\\
		J= & \left\{0, 4, 6\right\}
	\end{align*}
	so that 
	$$\Delta_{IJ}\left(\phi^3_9(F)\right)=\det\left(\begin{array}{ccc}
		c_3 & c_7 & c_9\\
		c_2 & c_6 & c_8\\
		c_0 & c_4 & c_6\\
	\end{array}\right).$$
	Then according to Equations \eqref{eq:lambda0} and \eqref{eq:mu0} above, we have
	$$\lambda=(7, 7, 6), \ \ \mu=(4, 1,0) \ \Rightarrow \ \mu+a=(5,2,1).$$
	The Young diagram for the skew shape $\lambda/(\mu+a)=(7,7,6)/(5,2,1)$ is 
	$$\young(:::::\hfil \hfil ,::\hfil \hfil \hfil \hfil \hfil,:\hfil \hfil \hfil \hfil \hfil).$$
	The set of Littlewood-Richardson tableaux on $\lambda/(\mu+a)$ are given by 
	$$\young(:::::00,::00011,:01112), \young(:::::00,::00011,:01122), \young(:::::00,::00011,:11122)$$
	corresponding respectively to Littlewood-Richardson partitions
	$$\nu_1=(6,5,1), \nu_2=(6,4,2), \nu_3=(5,5,2);$$
	note that the Littlewood-Richardson coefficients all equal one here.
	These partitions, in turn correspond to the column indexing sets
	$$K_1=\{0, 5, 7\}, \ K_2=\{1, 4, 7\}, \ K_3=\{1, 5, 6\}$$
	according to Equation \eqref{eq:nu},
	which define maximal minors of the shorter Toeplitz matrix 
	$$\phi^2_9(F)=\left(\begin{array}{cccccccc}
		c_2 & c_3 & c_4 & c_5 & c_6 & c_7 & c_8 & c_9\\
		c_1 & c_2 & c_3 & c_4 & c_5 & c_6 & c_7 & c_8\\
		c_0 & c_1 & c_2 & c_3 & c_4 & c_5 & c_6 & c_7\\ \end{array}\right)$$
	Then Lemma \ref{lem:LRr} tells us that 
	$$\Delta_{IJ}(\phi^3_9(F))=c_1^{3}\cdot s_{\lambda/(\mu+a)}=c_1^3\cdot \sum_{\nu}c^{\lambda}_{(\mu+a), \nu}\cdot s_\nu=\Delta_{K_1}(\phi^2_9(F))+\Delta_{K_2}(\phi^2_9(F))+\Delta_{K_3}(\phi^2_9(F))$$
	or 
	$$\det\left(\begin{array}{ccc}
		c_3 & c_7 & c_9\\
		c_2 & c_6 & c_8\\
		c_0 & c_4 & c_6\\
	\end{array}\right)=\det\left(\begin{array}{ccc}
		c_2 & c_7 & c_9\\
		c_1 & c_6 & c_8\\
		c_0 & c_5 & c_7\\ \end{array}\right)+\det\left(\begin{array}{ccc} c_3 & c_6 & c_9\\ 
		c_2 & c_5 & c_8\\
		c_1 & c_4 & c_7\\ \end{array}\right)+
	\det\left(\begin{array}{ccc}
		c_3 & c_7 & c_8\\
		c_2 & c_6 & c_7\\ 
		c_1 & c_5 & c_6\\ \end{array}\right).$$
	%(don't forget that we are assuming that $c_0=0$ here--actually in this example, that does not seem to matter).
\end{example}

We can also reverse this construction for certain maximal minors of $\phi^r_d(F)$.  Specifically, given a partition $\nu=(\nu_0,\ldots,\nu_r)$ satisfying $i-a\leq \nu_0\leq  d-r-a$ and $\nu_r\geq r-a$, define two new partitions $\lambda=(\lambda_0,\ldots,\lambda_r)$ and $\mu=(\mu_0,\ldots,\mu_r)$ by 
\begin{align}
	\label{eq:lambda}
	\lambda_p= & \nu_0+a-\max\{(i-a)-\nu_p,0\}\\
	\label{eq:mu}
	\mu_q= & \nu_0+a-i-\max\{\nu_{r-q}-(i-a),0\}.
\end{align}
Then it is straightforward to check that $\lambda$ and $\mu$ are both partitions, and that  $\lambda_p-\mu_0-a\geq r-a$ for all $0\leq p\leq r$.  Note that the condition $i-a\leq \nu_0$ is needed so that $\mu_q\geq 0$ for all $0\leq q\leq r$.  We identify row and column indexing sets 
\begin{align*}
	I= & \{0\leq i_0<\cdots<i_r\leq i\}\\
	J= & \{0\leq j_0<\cdots<j_r\leq d-i\}
\end{align*} 
by the formulas 
\begin{align}
	\label{eq:i}
	i_p= & \max\{(i-a)-\nu_p,0\}+p, & 0\leq p\leq r\\
	\label{eq:j}
	j_q= & \max\{\nu_{r-q}-(i-a),0\}+q, & 0\leq q\leq r.
\end{align}
One can check that with these definitions, $\lambda$, $\mu$, $I$ and $J$ satisfy Equations \eqref{eq:lambda0} and \eqref{eq:mu0}. 
As above, we define the column indexing subset $K=\{0\leq k_0<\cdots,k_r=d-r\}$ by the formula
\begin{align}
	\label{eq:K}
	k_q= & \nu_{r-q}-(r-a)+q, & 0\leq q\leq r.
\end{align}
Note that by our assumption, $i-a\leq \nu_0\leq d-r-a$, it follows that $i\leq k_r\leq d-r$.

Let us define the \emph{$\alpha^i$-statistic} on the partition $\nu$ (or on the set $K$) by the formula
\begin{equation}
	\label{eq:alpha}
	\alpha^i(\nu)=\alpha^i(K)=\sum_{q=0}^r\max\{(i-a)-\nu_q,0\}=\sum_{q=0}^r\max\{(i-r)-(k_q-q),0\}.
\end{equation}
Note that for partitions $\nu=(\nu_0,\ldots,\nu_r)$ satisfying $\nu_0\geq (i-a)$, if $t$, satisfying $0\leq t\leq r$, is the largest index for which $\nu_t\geq (i-a)$, then 
\begin{equation}
	\label{eq:alpha2}
	\alpha^i(\nu)=\sum_{q=t+1}^r\left((i-a)-\nu_q\right).
\end{equation}
Alternatively, $\alpha^i(\nu)$ is the size of the complimentary partition to $\lambda$ inside the $(r+1)\times (\nu_0+a)$ rectangle; see Figure \ref{fig:LambdaMu}.
Then we can prove the following:
\begin{lemma}
	\label{lem:alpha}
	Fix integers $a,d,i,r$ satisfying $0\leq a\leq r\leq i\leq \flo{d}$, and fix $r+1$-subsets $I$, $J$ and $K$ defined by Equations \eqref{eq:i}, \eqref{eq:j}, and \eqref{eq:K} above, with $i\leq k_r\leq d-r$.  Then in the Littlewood-Richardson expansion of the minor $\Delta_{IJ}(\phi^i_d(F))$, as in Lemma \ref{lem:LRr}, the minor $\Delta_K(\phi^r_d(F))$ occurs with Littlewood-Richardson coefficient $1$ and maximal $\alpha^i(K)$, i.e.
	$$\Delta_{IJ}\left(\phi^i_d(F)\right)=\Delta_K\left(\phi^r_d(F)\right)+\sum_{\substack{\nu'\in\operatorname{LRP}(\lambda/(\mu+a))\\ \alpha^i(K')<\alpha^i(K)\\}}c^\lambda_{(\mu+a),\nu'}\cdot \Delta_{K'}\left(\phi^r_d(F)\right).$$
\end{lemma}
\begin{proof}
	For concreteness, let us introduce coordinates for the boxes of the Young diagram of $\lambda/(\mu+a)$ so that the lower left corner is at the origin.  Define the critical line to be the vertical line $x=\lambda_0-(i-a)$.  Our first observation is that for any Littlewood-Richardson tableau $T\in\operatorname{LRT}(\lambda/(\mu+a))$, the reverse lattice word condition as well as the column-strict condition implies that entries of $T$ lying to the right of the critical line are fixed, i.e. the $p^{th}$ row must contain only $p$'s for $0\leq p\leq r$.  Moreover, if $t$, $0\leq t\leq r$ is the largest index for which $\nu_t\geq (i-a)$, then we claim that there is exactly one Littlewood-Richardson tableau $T$ whose entries to the left of the critical line all lie in $\{0,\ldots,t\}$.  To see this, start filling in the boxes to the left of the critical line, column by column moving top to bottom and from right to left.  Note that the righter-most column to the left of the critical line has $t_1=\#\left\{p \ | \ \nu_p\geq (i-a)+1\right\}$ boxes where $t_1\leq t$.  Hence, by the reverse lattice and column-strict properties, the entries in that column must be exactly $\{0,\ldots,t_1-1\}$.  Inductively, the $q^{th}$ righter-most column to the left of the critical line has $t_q=\#\{p \ | \ \nu_p\geq (i-a)+q\}$ boxes where $t_q\leq t_{q-1}\leq \cdots\leq t_1\leq t$.  Moreover the entries in those boxes can only be $\{0,\ldots,t_{q-1}-1\}$ by the weakly increasing row property.  Again by the reverse lattice word property, it follows that the boxes in the $q^{th}$ righter-most column to the left of the critical line must be filled with the numbers $\{0,\ldots,t_q-1\}$.  This shows that the entries of $T\in\operatorname{LRT}_\nu(\lambda/(\mu+a))$ are uniquely determined, and hence $c^\lambda_{(\mu+a),\nu}=1$.

	Finally, if $\nu'$ is any other Littlewood-Richardson partition for $\lambda/(\mu+a)$ distinct from $\nu$, then it must differ in the entries $\{t+1,\ldots,r\}$ and hence we must have $\nu'_j\geq \nu_j$ for each $j\in\{t+1,\ldots,r\}$ with strict inequality for at least one $j$.  Therefore, it follows from Equation \eqref{eq:alpha2} that we must have $\alpha^i(K')=\alpha^i(\nu')<\alpha^i(\nu)=\alpha^i(K)$, and the second statement follows from Lemma \ref{lem:LRr}.
	%Then the Littlewood-Richardson condition implies that in any Littlewood-Richardson tableau $T'$ corresponding to $\nu'$, its entries to the right of the vertical line $x=\nu_0-(i-a)$ must be the same as that of $\nu$, namely the $k^{th}$ row must consist of only $k$'s to the right of that line.  Let $t$, $0\leq t\leq r$ be the largest index for which $\nu_t-(i-a)\geq 0$.  Then we further observe that for any Littlewood-Richardson filling of the Young diagram of $\lambda/(\mu+a)$, if the entries in the boxes that lie to the left of the line $x=d-r-(i-a)$ use only the numbers $\{0,\ldots,t\}$, then that filling is unique and its content must be $\nu$, again by the Littlewood-Richardson condition.  Therefore if $\nu'$ and $\nu$ are different, they must differ by their entries in $\{t+1,\ldots,r\}$.  On the other hand, if $\nu'$ has numbers from $\{t+1,\ldots,r\}$ to the left of the line $x=d-r-(i-a)$ then its $\alpha^i$-statistic must be strictly smaller than that of $\nu$, since the numbers in $\{t+1,\ldots,r\}$ only occur to the right of the line $x=d-r-(i-a)$ in the filling corresponding to $\nu$.  
	See Figure \ref{fig:LambdaMu}.

	\begin{figure}
		\begin{center}
			\begin{tikzpicture}[scale=1]
				% Axes
				\draw[->, thick] (-0.5,0) -- (11,0) node[right] {$x$};
				\draw[->, thick] (0,-0.5) -- (0,5) node[above] {$y$};
				% Origin
				\node[below left] at (0,0) {$0$};
				% X-axis points
				\coordinate (a) at (2,0);
				\coordinate (b) at (8,0);
				\coordinate (lam) at (10.4,0);
								
				\node[below] at (a) {$a$};
				\node[below] at (lam) {$\lambda_0$};
				\node[below] at (b) {$\lambda_0-(i-a)$};
				
				% Y-axis points
				\coordinate (rp1) at (0,1.8);
				\coordinate (rmt) at (0,4.2);
				
				\node[left] at (rp1) {$r-t$};
				\node[left] at (rmt) {$r+1$};
				
				\node[above] at (8,5) {critical line};
				
				% Dashed vertical lines (first quadrant only)
				\draw[dashed] (2,0) -- (2,5);
				\draw[dashed] (8,0) -- (8,5);
				\draw[dashed] (10.4,0) -- (10.4,5);

				% Dashed horizontal lines (first quadrant only)
				\draw[dashed] (0,1.8) -- (8,1.8);
				\draw[dashed] (0,4.2) -- (8,4.2);
				
				% ---- BOX STRIP ABOVE X-AXIS ----
				\def\boxsize{0.6}
				\def\ybox{0} % vertical offset above x-axis
				
				% Starting x at point a
				\def\xstart{2}
				
				% First 4 boxes labeled 0
				\foreach \i in {0,...,3} {
					\draw (\xstart+\i*\boxsize, \ybox)
					rectangle ++(\boxsize,\boxsize);
					\node at (\xstart+\i*\boxsize+0.5*\boxsize, \ybox+0.5*\boxsize) {$0$};
				}
				\foreach \i in {4,...,7} {
					\draw (\xstart+\i*\boxsize, \ybox+\boxsize)
					rectangle ++(\boxsize,\boxsize);
					\node at (\xstart+\i*\boxsize+0.5*\boxsize, \ybox+1.5*\boxsize) {$0$};
				}
				
				\foreach \i in {10,...,13} {
					\draw (\xstart+\i*\boxsize, \ybox+6*\boxsize)
					rectangle ++(\boxsize,\boxsize);
					\node at (\xstart+\i*\boxsize+0.5*\boxsize, \ybox+6.5*\boxsize) {$0$};
				}
				
				% Next 4 boxes labeled 1
				
				\foreach \i in {10,...,13} {
					\draw (\xstart+\i*\boxsize, \ybox+5*\boxsize)
					rectangle ++(\boxsize,\boxsize);
					\node at (\xstart+\i*\boxsize+0.5*\boxsize, \ybox+5.5*\boxsize) {$1$};
				}
				
				% Next 4 boxes labeled \vdots
				\foreach \i in {4,...,7} {
					\draw (\xstart+\i*\boxsize, \ybox)
					rectangle ++(\boxsize,\boxsize);
					\node at (\xstart+\i*\boxsize+0.5*\boxsize, \ybox+0.5*\boxsize) {$\vdots$};
				}
				
				\foreach \i in {8,...,12} {
					\draw (\xstart+\i*\boxsize, \ybox+\boxsize)
					rectangle ++(\boxsize,\boxsize);
					\node at (\xstart+\i*\boxsize+0.5*\boxsize, \ybox+1.5*\boxsize) {$\vdots$};
				}
				
				\foreach \i in {10,...,13} {
					\draw (\xstart+\i*\boxsize, \ybox+4*\boxsize)
					rectangle ++(\boxsize,\boxsize);
					\node at (\xstart+\i*\boxsize+0.5*\boxsize, \ybox+4.5*\boxsize) {$\vdots$};
				}
				
				% Next 4 labeled t_2, t_1, r, r
				\def\labels{{t_2},{t_1},{r},{r}}
				\def\xdotstart{\xstart+7*\boxsize}
				\foreach \lab [count=\i] in \labels {
					\draw (\xdotstart+\i*\boxsize, \ybox)
					rectangle ++(\boxsize,\boxsize);
					\node at (\xdotstart+\i*\boxsize+0.5*\boxsize, \ybox+0.5*\boxsize) {$\lab$};
				}
				\def\labels{0,1,{$\small{t+1}$},{$\small{t+1}$},{$\small{t+1}$}}
				\def\xdotstart{\xstart+7*\boxsize}
				\foreach \lab [count=\i] in \labels {
					\draw (\xdotstart+\i*\boxsize, \ybox+2*\boxsize)
					rectangle ++(\boxsize,\boxsize);
					\node at (\xdotstart+\i*\boxsize+0.5*\boxsize, \ybox+2.5*\boxsize) {$\lab$};
				}
				\def\labels{0,{$t$},{$t$},{$t$},{$t$}}
				\def\xdotstart{\xstart+8*\boxsize}
				\foreach \lab [count=\i] in \labels {
					\draw (\xdotstart+\i*\boxsize, \ybox+3*\boxsize)
					rectangle ++(\boxsize,\boxsize);
					\node at (\xdotstart+\i*\boxsize+0.5*\boxsize, \ybox+3.5*\boxsize) {$\lab$};
				}

				% Two final boxes without margins containing dots, ending at x = lambda_0
				\def\xdotstart{\xstart+12*\boxsize}
				\foreach \i in {0,1} {
					% \draw (\xdotstart+\i*\boxsize, \ybox)
					%       rectangle ++(\boxsize,\boxsize);
					\fill (\xdotstart+\i*\boxsize+0.5*\boxsize, \ybox+0.5*\boxsize) circle (2pt);
				}
				\def\xdotstart{\xstart+13*\boxsize}
				\fill (\xdotstart+0.5*\boxsize, \ybox+1.5*\boxsize) circle (2pt);
				\fill (\xdotstart+0.5*\boxsize, \ybox+2.5*\boxsize) circle (2pt);
			\end{tikzpicture}
		\end{center}
		\caption{The unique $\nu$-filling of the skew shape $\lambda/(\mu+a)$.  The dots in the lower right corner count the $\alpha^i$ statistic.}
		\label{fig:LambdaMu}
	\end{figure}
\end{proof}
In the next section we will see the significance of the $\alpha^i$-statistic of a partition; it represents the order of vanishing of a certain specialization of the mixed Hessian determinant. 

\section{Mixed Hessian determinants and weighted NE lattice paths}
\label{sec:MixedHessians}
According to \cite[Lemma 4.19]{MMS}, for any $F\in Q_d$ and for any $0\leq r\leq s(F)-1$, the permuted $r^{th}$ mixed Hessian satisfies the following factorization formula
\begin{equation}
	\label{eq:factor}
	P_r\cdot \operatorname{MHess}_r(F,\mathcal{E})|_{(\underline{X},\underline{Y})}=d!\cdot \left(\phi_d(F)\cdot W_{d-2r}(\underline{X},\underline{Y})\right)_{IJ}	
\end{equation}
where $\phi_d(F)=(c_{q-p})_{0\leq p,q}$ is the bi-infinite Toeplitz matrix of $F$, $$W_{d-2r}(\underline{X},\underline{Y})=W_{d-2r}(X_1,\ldots,X_{d-2r},Y_1,\ldots,Y_{d-2r})$$ is the bi-infinite weighted path matrix corresponding to the NE lattice paths from the infinite vertex set $\mathcal{A}=\left\{(-p,p) \ | \ 0\leq p\right\}$ to the infinite vertex set $\mathcal{B}=\left\{(-q,d-2r+q) \ | \ 0\leq q\right\}$ where the North, respectively East edges emanating from the diagonal line $x+y=p$ are weighted by the variables $Y_{p+1}$, respectively $X_{p+1}$, for $0\leq p\leq d-2r-1$, and where $I=\left\{0,\ldots,r\right\}$ and $J=\left\{r,\ldots,2r\right\}$, and $P_r$ is the $(r+1)\times (r+1)$ permutation matrix for the order-reversing permutation $p\mapsto r-p$, $0\leq p\leq r$.  

The following so-called Pl\"ucker expansion formula for the permuted mixed Hessian determinant follows from Equation \eqref{eq:factor} and is implicit in \cite[Theorem 4.21]{MMS}.  We carefully state and prove it here as a lemma for future reference.
\begin{lemma}
	\label{lem:Plucker}
	For any $F\in Q_d$ with Sperner number $s=s(F)$, and for any $0\leq r\leq s-1$, we have 
\begin{equation}
	\label{eq:detMixed}
	\det\left(P_r\cdot \operatorname{MHess}_r(F,\mathcal{E})|_{(\underline{X},\underline{Y})}\right)=(d!)^{r+1}\cdot \sum_{K\in\binom{[d-r]_0}{r+1}}\Delta_K\left(\phi^r_d(F)\right)\cdot \Delta_{(K+r)J}\left(W_{d-2r}\left(\underline{X},\underline{Y}\right)\right).
\end{equation}
Moreover the determinant 
\begin{equation}
\label{eq:Pathsystem}	\Delta_{(K+r)J}\left(W_{d-2r}(\underline{X},\underline{Y})\right)=\sum_{\mathcal{P}\colon \mathcal{A}_{K+r}\rightarrow \mathcal{B}_J}\prod_{z=1}^{d-2r}X_z^{a_z(\mathcal{P})}Y_z^{b_z(\mathcal{P})}
\end{equation} 
is the sum of monomials indexed by vertex disjoint NE lattice path systems from $\mathcal{A}_{K+r}=\left\{A_{k_q}=(-k_q-r,k_q+r) \ | \ 0\leq q\leq r\right\}$ to $\mathcal{B}_J=\left\{B_{r+q}=(-q-r, d-r+q) \ | \ 0\leq q\leq r\right\}$, where $a_z(\mathcal{P})$, respectively $b_z(\mathcal{P})$, is the number of East, respectively North, edges of $\mathcal{P}$ emanating from the diagonal $x+y=z-1$, for $1\leq z\leq d-2r$.
\end{lemma}
\begin{proof}
	It follows from the factorization formula in Equation \eqref{eq:factor} and from the generalized Cauchy-Binet theorem that the determinant of the permuted $r^{th}$ mixed Hessian matrix satisfies
	\begin{equation}
		\label{eq:pluck1}
		\det\left(P_r\cdot \operatorname{MHess}_r(F,\mathcal{E})|_{(\underline{X},\underline{Y})}\right)=(d!)^{r+1}\cdot \sum_{K\in\binom{\N_0}{r+1}}\Delta_{IK}\left(\phi_d(F)\right)\cdot \Delta_{KJ}\left(W_{d-2r}\left(\underline{X},\underline{Y}\right)\right).
	\end{equation}
	where the sum is over all $r+1$-subsets $K=\{0\leq k_0<\cdots<k_r\}\subset\N_0$ of natural numbers, including $0$.  Since $W_{d-2r}(\underline{X},\underline{Y})$ is a weighted path matrix, it follows from the Lindst\"om-Gessel-Viennot theorem \cite[Fact 4.17]{MMS} that its $KJ$-minor is the sum, overall all vertex disjoint NE lattice path systems from $\mathcal{A}_K=\{(-k_q,k_q) \ | \ 0\leq q\leq r\}$ to $\mathcal{B}_J=\{(-q-r,d-r+q) \ | \ 0\leq q\leq r\}$, of the product of edge weights in the path system, which is Equation \eqref{eq:Pathsystem}.  For $(r+1)$-subsets $K\not\subset\{r,\ldots,d\}$ (i.e. $k_0<r$), $\Delta_{KJ}(W_{d-2r}(\underline{X},\underline{Y}))\equiv 0$, since there are no vertex disjoint NE lattice path systems from $\mathcal{A}_K$ to $\mathcal{B}_J$ in that case, cf. \cite[Lemma 4.18]{MMS}.  On the other hand, for $K\subset\{r,\ldots,d\}$ (i.e. $r\leq k_0$), we have an equality of minors 
	$$\Delta_{IK}(\phi_d(F))=\Delta_{K-r}(\phi^r_d(F))$$
	where $K-r\subset \{0,\ldots,d-r\}$ is the subset $K$ shifted down by $r$, indexing columns of the finite Toeplitz matrix $\phi^r_d(F)$.  Putting these two observations together, Equation \eqref{eq:detMixed} follows.
\end{proof}

%Recall our Pl\"ucker expansion formula for the determinant of the permuted $r^{th}$ mixed Hessian determinant:

%where $W_{d-2r}\left(\underline{X},\underline{Y}\right)$ is the bi-infinite weighted path matrix for the NE lattice paths from the infinite vertex set $\mathcal{A}=\left\{(-p,p) \ | \ 0\leq p\right\}$ to the infinite vertex set $\mathcal{B}=\left\{(-q,d-2r+q) \ | \ 0\leq q\right\}$ with North, respectively East edges issuing from diagonal line $x+y=p$ weighted with the variables $Y_{p+1}$, respectively $X_{p+1}$ for $0\leq p\leq d-2r-1$, and where $J=\left\{r,\ldots,2r\right\}$ and $K+r=\left\{r\leq k_0+r<\cdots<k_r+r\leq d\right\}$ is the index $K$ shifted up by $r$.  It follows from \cite[Fact 4.17]{MMS} that the minor determinant   

%a homogeneous polynomial in the variables $(\underline{X},\underline{Y})=(X_1,\ldots,X_{d-2r},Y_1,\ldots,Y_{d-2r})$ whose monomial terms are the product of edge weights of vertex disjoint NE lattice path systems from $\mathcal{A}_{K+r}=\left\{(-k_q-r,k_q+r) \ | \ 0\leq q\leq r\right\}$ to $\mathcal{B}_J=\left\{(-p-r,d-r+p) \ | \ 0\leq p\leq r\right\}$
%where the North, respectively East, edges issuing from the diagonal line $x+y=p$ are weighted by the variables $Y_{p+1}$, respectively $X_{p+1}$, for $0\leq p\leq d-2r-1$.

\begin{example}
	\label{ex:WPM}
	Let $d=6$ and $r=2$ so that $d-r=4$ and  $d-2r=2$.  Then $J=\left\{2,3,4\right\}$.  The weighted path determinants are given in the table below (together with the $\alpha^i$-statistics for $i=3$):
	$$\begin{array}{l|l|l|r}
		K & K+r & \Delta_{(K+r)J}\left(W_2(X_1,X_2,Y_1,Y_2)\right) & \alpha^i(K)\\
		\hline
		&&& \\
		\left\{0,1,2\right\} & \left\{2,3,4\right\} & Y_1^3Y_2^3 & 3\\
		&&\\
		\left\{0,1,3\right\} & \left\{2,3,5\right\} & X_1Y_1^2Y_2^3+X_2Y_1^3Y_2^2 & 2\\
		&&&\\
		\left\{0,1,4\right\} & \left\{2,3,6\right\} & X_1X_2Y_1^2Y_2^2 & 2\\
		&&&\\
		\left\{0, 2, 3\right\} & \left\{2,4,5\right\} & X_1^2Y_1Y_2^3+X_1X_2Y_1^2Y_2^2+X_2^2Y_1^3Y_2 & 1\\
		&&&\\
		\left\{0,2,4\right\} & \left\{2,4,6\right\} & X_1^2X_2Y_1Y_2^2+X_1X_2^2Y_1^2Y_2 & 1\\
		&&&\\
		\left\{0,3,4\right\} & \left\{2,5,6\right\} & X_1^2X_2^2Y_1Y_2 & 1\\
		&&&\\
		\left\{1, 2, 3\right\} & \left\{3,4,5\right\} & X_1^3Y_2^3+X_1^2X_2Y_1Y_2^2+X_1X_2^2Y_1^2Y_2+X_2^3Y_1^3 & 0\\
		&&&\\
		\left\{1, 2, 4\right\} & \left\{3,4,6\right\} & X_1^3X_2Y_2^2+{X_1^2X_2^2Y_1Y_2}+X_1X_2^3Y_1^2 & 0\\
		&&&\\
		\left\{1, 3, 4\right\} &  \left\{3,5,6\right\} & X_1^3X_2^2Y_2+X_1^2X_2^2Y_1 & 0\\
		&&&\\
		\left\{2, 3, 4\right\} & \left\{4,5,6\right\} & X_1^3X_2^3 & 0\\
	\end{array}$$
	Figure \ref{fig:WPM} shows the vertex disjoint path system corresponding to the monomial term $X_1^2X_2^2Y_1Y_2$ of $\Delta_{(K+r)J}(W_2(X_1,X_2,Y_1,Y_2))$ for $K+r=\{3,4,6\}$, with $d$, $r$, and $J$ as above.
	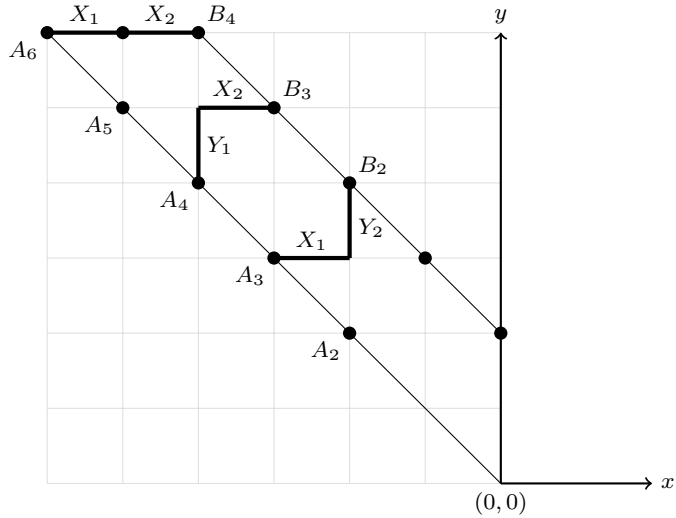
\begin{figure}
		\begin{center}
			\begin{tikzpicture}[scale=1]
			% Draw 6x6 grid
				\draw[step=1cm, gray, very thin] (0,0) grid (6,6);
				
				% Axes at lower-right corner (6,0)
				\draw[->, thick] (6,0) -- (8,0) node[right] {$x$};
				\draw[->, thick] (6,0) -- (6,6) node[above] {$y$};
				
				% Main diagonal
				\draw[] (0,6) -- (6,0);
				
				% Parallel diagonal
				\draw[] (2,6) -- (6,2);
				
				% Points on main diagonal
				\foreach \x/\y/\lab in {
					0/6/A_6,
					1/5/A_5,
					2/4/A_4,
					3/3/A_3,
					4/2/A_2
				}{
					\fill (\x,\y) circle (2.5pt);
					\node[below left] at (\x,\y) {$\lab$};
				}
				
				% Points on parallel diagonal
				\foreach \x/\y/\lab in {
					2/6/B_4,
					3/5/B_3,
					4/4/B_2,
					5/3/,
					6/2/ 
				}{
					\fill (\x,\y) circle (2.5pt);
					\node[above right] at (\x,\y) {$\lab$};
				}
				
				% Thick segment from A_6 (0,6) to B_4 (3,5)
				\draw[ultra thick] (0,6) -- (2,6);
				\draw[ultra thick] (2,4) -- (2,5);
				\draw[ultra thick] (2,5) -- (3,5);
				\draw[ultra thick] (3,3) -- (4,3);
				\draw[ultra thick] (4,3) -- (4,4);

				% Midpoint of segment
				\fill (1,6) circle (2.5pt);
				
				% Labels X_1 and X_2 on segment halves
				\node[above] at (0.5,6) {$X_1$};
				\node[above] at (1.5,6) {$X_2$};
				\node[right] at (2,4.5) {$Y_1$};
				\node[above] at (2.4,5) {$X_2$};
				\node[above] at (3.5,3) {$X_1$};
				\node[right] at (4,3.4) {$Y_2$};
				\node[below] at (6,0) {$(0,0)$};
			\end{tikzpicture}
		\end{center}
		\caption{vertex disjoint path system for the monomial term $X_1^2X_2^2Y_1Y_2$ of $\Delta_{(K+r)J}(W_2(X_1,X_2,Y_1,Y_2))$ for $K+r=\{3,4,6\}$ and $J=\{2,3,4\}$}
		\label{fig:WPM} 
	\end{figure}
	Note that if we take $i=3>r=2$ so that $i-r=1$, then specializing the variables to $Y_1=t$ and $X_1,X_2,Y_2=1$ yields polynomials $\Delta_{(K+r)J}\left(W_2(1,1,t,1)\right)$ whose order of vanishing at $t=0$ is the $\alpha^i$ statistic of $K$. 
\end{example} 
We can generalize the last statement of Example \ref{ex:WPM}.
\begin{lemma}
	\label{lem:Specialization}
	Fix $a,d,i,r$ as in Lemma \ref{lem:alpha}, with $J=\{r,\ldots,2r\}$ and $K\in\binom{[d-r]_0}{r+1}$ as in Lemma \ref{lem:Plucker}.  Let $\Delta^i_{(K+r)J}(t)$ be the specialization of $\Delta_{(K+r)J}\left(W_{d-2r}(\underline{X},\underline{Y})\right)$ at the point where $Y_1=\cdots=Y_{i-r}=t$, and the remaining variables are equal to $X_i=1=Y_j$.  Then the order of vanishing of univariate polynomial $\Delta^i_{(K+r)J}(t)$ at $t=0$ is equal to the $\alpha^i$-statistic of $K$, $\alpha^i(K)$.	 
\end{lemma}
\begin{proof}
	First note that we can realize our specialization $\Delta^i_{(K+r)J}(t)$ as the determinant of the weighted path matrix from $\mathcal{A}_{K+r}$ to $\mathcal{B}_J$ by specializing the weights of our edges in our NE lattice path graph so that the first $(i-r)$ North steps, i.e. the North steps emanating from the diagonal lines $x+y=p$ for $0\leq p\leq i-r-1$, are weighted with $t$, and the remaining edges are weighted with $1$.  Then we must identify a vertex disjoint NE lattice path system from $\mathcal{A}_{K+r}$ to $\mathcal{B}_J$ that minimizes the number of $t$-weighted North edges.  One such minimizing vertex disjoint path system is the one in which all East steps precede all North steps in every path.  Since the number of East steps from $A_{k_q+r}=(-k_q-r,k_q+r)$ to $B_{r+q}=(-r-q,d-r+q)$ is $k_q-q$, it follows that the power of $t$ coming from this path is equal to 
	$\max\{(i-r)-(k_q-q),0\}$, and hence summing over $q$ yields the minimal power of $t$ in the polynomial $\Delta^i_{(K+r)J}(t)$, and the result follows from Equation \eqref{eq:alpha}.	
\end{proof}

\section{Proof of Theorem \ref{thm:Lorentzian2}}
\label{sec:Theorem1}
We are now in a position to prove Theorem \ref{thm:Lorentzian2}.  We restate it here for the convenience of the reader.
\begin{theorem}
	\label{thm:Lorentzian21}
	Let $F\in Q_d$ be a homogeneous $d$-form, and fix $0\leq i\leq \flo{d}$.	The following are equivalent.
	\begin{enumerate}
		\item $F$ is $i$-Lorentzian
		\item $\phi^i_d(F)$ is strongly totally nonnegative.
		\item $A_F$ satisfies mixed HRR$_i$ on the standard open convex cone of linear forms
		$$U=\left\{ax+by \ | \ (a,b)\in\R^2_{<0}\right\}.$$
	\end{enumerate}
\end{theorem}

\begin{proof}	
	(1) $\Rightarrow$ (2).  Assume that $F\in L(i)_d$ is $i$-Lorentzian.  Then, by definition, there exists a sequence $\{F_n\}\subset \Lor(i)_d$ of strictly $i$-Lorentzian polynomials such that $\lim_{n\to\infty}F_n=F$.  By \cite[Theorem 1]{MMS} (or more specifically, \cite[Proposition 4.8]{MMS}), the Toeplitz matrices $\phi^i_d(F_n)$ must be totally positive for all $n$, and hence by Lemma \ref{lem:super}, so then are the matrices $\phi^j_d(F_n)$ for all $n$ and for all $j$, $0\leq j\leq i$.  By continuity we have $\lim_{n\to\infty}\phi^j_d(F_n)=\phi^j_d(F)$,and hence it follows that $\phi^j_d(F)$ must be totally nonnegative for every $0\leq j\leq i$, which is (2).
	
	(2) $\Rightarrow$ (3).  Assume that $\phi^i_d(F)$ is strongly totally nonnegative.  Then, in particular, for each $j$, with $0\leq j\leq i$, every maximal minor of $\phi^j_d(F)$ is totally nonnegative, and hence it follows from Lemma \ref{lem:Plucker} that  
	$$\det\left(P_j\cdot\operatorname{MHess}_j(F,\mathcal{E})|_{(\underline{X},\underline{Y})}\right)=\sum_{K\in\binom{[d-j]_0}{j+1}}\Delta_K(\phi^j_d(F))\cdot\left(\sum_{\mathcal{P}\colon \mathcal{A}_{K+j}\rightarrow\mathcal{B}_J}\prod_{z=1}^{d-2j}X_z^{a_z(\mathcal{P})}\cdot Y_z^{b_z(\mathcal{P})}\right)>0,$$
	for all $(\underline{X},\underline{Y})=(X_1,\ldots,X_{d-2j},Y_1,\ldots,Y_{d-2j})>0$, and for each $j$ with $0\leq j\leq \min\{i,s-1\}$.
	Therefore it follows from \cite[Lemma 3.13]{MMS} that the algebra $A_F$ must satisfy mixed HRR$_i$ on $U$, which is (3).

	(3) $\Rightarrow$ (1).  Fix $i$ with $0\leq i\leq \flo{d}$.  We will show that if $A_F$ satisfies the mixed HRR$_i$ on $U$, then  $F\in L(i)_d$, by downward induction on the Sperner number $s=s(F)$.  For the base case, assume that  $i+1\leq s(F)$, so that rank of $\phi^i_d(F)$ is $i+1$, by \cite[Lemma 4.5]{MMS}.  Then note that $\ell_1=x+ty$ and $\ell_2=tx+y$ are linearly independent and in $U$ for all $t$, with $0<t<1$.  It follows from \cite[Corollary 4.13]{MMS} that $G_t(X,Y)=F(X+tY,tX+Y)$ is strictly $i$-Lorentzian.  Since $\lim_{t\to 0}G_t=F$, it follows that $F\in L(i)_d$ which completes the base case.
	
	For the inductive step, assume the result holds for all homogeneous $d$-forms $H$ with Sperner number satisfying $i+1\geq s(H)>s$, and let $F\in Q_d$ be a homogeneous $d$-form with Sperner number $s(F)=s$, and such that $A_F$ satisfies the mixed HRR$_i$ on $U$. Note that $G_t=F(X+tY,tX+Y)$ from above cannot be strictly $i$-Lorentzian because $s(G_t)-1=s-1< i$.  On the other hand, since the algebra $A_F$ satisfies the mixed HRR$_i$ on $U$, it follows that the isomorphic algebra $A_{G_t}$ satisfies mixed HRR$_i$ (and hence also mixed HRR$_{s-1}$) on the standard closed cone $\overline{U}$; this argument is implicit in the proof of \cite[Corollary 4.13]{MMS}.  It therefore follows from \cite[Theorem 1]{MMS} that $\phi^{s-1}_d(G_t)$ is totally positive, and hence by Lemma \ref{lem:super} as well as Lemma \ref{lem:TPs}, that $\phi^{j}_d(G_t)$ is either totally positive if $0\leq j\leq s-1$ or TP$_s$ and totally nonnegative if $s\leq j\leq i$.  In order to invoke our induction hypothesis, we will find another family of polynomials converging to $G_t$ with similar positivity properties.  
	
	To this end, let $u$ be another parameter and define the two-parameter family of polynomials by 
	$$H_{t,u}=G_t+(-1)^suY^d.$$
	Then for all $j$, with $0\leq j\leq\flo{d}$, we have 
	$$\phi^j_d(H_{t,u})=\phi^j_d(G_t)+(-1)^suE_{0,d-2j}$$
	where $E_{0,d-j}$ is the $(j+1)\times (d-j+1)$ elementary matrix with $1$ in the $(0,d-j)$ entry (upper right corner) and zeros elsewhere.  It follows that for any $k$, with $1\leq k\leq j+1$, and for any $k$-subsets $A\subset\{0,\ldots,j\}$ and $B\subset\{0,\ldots,d-j\}$, we have  
	$$\Delta_{AB}\left(\phi^j_d(H_{t,u})\right)=\begin{cases} \Delta_{AB}\left(\phi^j_d(G_t)\right) & \text{if} \ 0\notin A \ \text{or} \ d-j\notin B\\
		\Delta_{AB}\left(\phi^j_d(G_t)\right)+(-1)^{s+k-1}u\Delta_{(A\setminus 0) (B\setminus d-j)}\left(\phi^j(G_t)\right) & \text{if} \ 0\in A \ \text{and} \ d-j\in B\\ \end{cases}$$
	In particular it is clear that for all $0<t<1$ and ${u>0}$ sufficiently small (possibly depending on $t$), $\phi^j_d(H_{t,u})$ is totally positive for $0\leq j\leq s-1$, and TP$_s$ and totally nonnegative for $s\leq j\leq i$, since $\phi^j_d(G_t)$ is.  For $j$ satisfying $s\leq j\leq i$, and $k=s+1$, the $k\times k$ minor $\Delta_{AB}(\phi^j_d(H_{t,u}))$ is either zero, corresponding to the $(s+1)\times (s+1)$ minor $\Delta_{AB}(\phi^j_d(G_t))$ or else it is positive, being a $u$-multiple of the $s\times s$ minor $\Delta_{(A\setminus 0)(B\setminus d-j)}(\phi^j_d(G_t))$.  In particular, the matrix $\phi^s_d(H_{t,u})$ is totally nonnegative of rank $s+1$.  Moreover, since for $k>s+1$, all $k\times k$ minors $\Delta_{AB}(\phi^j_d(H_{t,u}))$ vanish, it follows that the matrix $\phi^{s+1}_d(H_{t,u})$ does not have maximal rank, and hence by \cite[Lemma 4.5]{MMS}, the Sperner number of $H_{t,u}$ is $s(H_{t,u})=s+1$.  In order to invoke our induction hypothesis, it remains show that the algebra $A_{H_{t,u}}$ satisfies mixed HRR$_i$ on $U$.

%	For ${j=s}$, the maximal minors (of size $k=s+1$) of $\phi^s_d(H_{t,u})$ are either zero, corresponding to the maximal minors of $\phi_d^s(G_t)$, or else they are positive, being a $u$-multiple of an $s\times s$ minor of the TP$_s$ matrix $\phi^{s}_d(G_t)$.  In particular, the maximal minors of the Toeplitz matrix $\phi^s_d(H_{t,u})$ are nonnegative, with at least one of them being positive.   Moreover, since the maximal minors (of size $k=s+2$) of $\phi^{s+1}_d(H_{t,u})$ are all zero it follows that the Sperner number of $H_{t,u}$ satisfies $s(H_{t,u})=s+1$.  Hence in order to invoke our induction hypothesis, it remains to show that the algebra $A_{H_{t,u}}$ satisfies mixed HRR$_i$ on $U$.   
	%Since the minors of the $(j+1)\times (j+1)$ matrix $\phi^j_d(H_{t,u})$ are also minors of the matrix $\phi^i_d(H_{t,u})$ for all $j$, with $0\leq j\leq i$, it follows that $\operatorname{rk}\bigl(\phi^i_d(H_{t,u})\bigr)=s+1$, and hence our induction hypothesis applies.  It remains to show that the oriented AG algebra $A_{H_{t,u}}$ satisfies mixed HRR$_i$ on $U$.  
	We proceed as above:  fix $j$, with $0\leq j\leq \min\{i,s(H_{t,u})-1\}=s$.  Then since $\phi^j_d(H_{t,u})$ is either totally positive, if $0\leq j<s$, or totally nonnegative with at least one nonvanishing maximal minor, if $j=s$, it follows from Lemma \ref{lem:Plucker} that
	$$\det\left(P_j\cdot\operatorname{MHess}_j(H_{t,u},\mathcal{E})|_{(\underline{X},\underline{Y})}\right)=\sum_{K\in\binom{[d-j]_0}{j+1}}\Delta_K(\phi^j_d(H_{t,u}))\cdot\left(\sum_{\mathcal{P}\colon \mathcal{A}_{K+j}\rightarrow\mathcal{B}_J}\prod_{z=1}^{d-2j}X_z^{a_z(\mathcal{P})}\cdot Y_z^{b_z(\mathcal{P})}\right)>0,$$
	for all $(\underline{X},\underline{Y})>0$.  Therefore it follows from \cite[Lemma 3.13]{MMS} that the algebra $A_{H_{t,u}}$ satisfies mixed HRR$_i$ on $U$.  
	%From Lemma 4.19 and the generalized Cauchy-Binet formula we have
	%\begin{align*}
	%	\det\left(P_j\operatorname{MHess}_j(H_{t,u},\mathcal{E})|_{\underline{C}}\right)= &  \sum_{K\in  \binom{\{j,\ldots,d\}}{j+1}} \det\left(\phi_d(H_{t,u})_{IK}\right)\det\left(W_{d-2j}(\underline{a},\underline{b})_{KJ}\right)\\
	%	= & \sum_{K\in\binom{\{j,\ldots,d\}}{j+1}}\det\left(\phi^j_d(H_{t,u})_{IK-j}\right)\det\left(W_{d-2j}(\underline{a},\underline{b})_{KJ}\right)
	%\end{align*}
	%where $I=\{0,\ldots,j\}$, $J=\{j,\ldots,2j\}$, and the sum is over all $(j+1)$-subsets $K\subset \{j,\ldots,d\}$, and $K-j\subset\{0,\ldots,d-j\}$ is the shifted set, indexing the columns in $\phi^j_d(H_{t,u})$.  In particular the determinant of the $j^\text{th}$ permuted mixed Hessian is a linear combination of all maximal minor determinants of $\phi^j_d(H_{t,u})$, which we have already showed are all nonnegative with at least one strictly positive.  It follows from Lemma 4.18(1) that for each    $\underline{C}=\left(\underline{a},\underline{b}\right)\in\left(\R^2_{>0}\right)^{d-2j}$, we have $\det\left(W_{d-2j}(\underline{a},\underline{b})\right)>0$, and hence that $\det\left(P_j\operatorname{MHess}_j(F,\mathcal{E})|_{\underline{C}}\right)>0$.  Since this holds for each $j$, with $0\leq j\leq \min\{i,s(H_{t,u})-1\}$, this implies that $A_{H_{t,u}}$ satisfies mixed HRR$_i$ on $U$ by Lemma 3.13.  
	Therefore by our induction hypothesis, $H_{t,u}$ is $i$-Lorentzian for all $0<t<1$ and $u=u(t)>0$ sufficiently small.  Finally, since $\lim_{(t,u)\to (0,0)}H_{t,u}=F$ it follows that $F$ is $i$-Lorentzian as well, which is (1).
\end{proof}

\section{Proof of Theorem \ref{thm:STN=TN}}
\label{sec:Theorem2}
We are now in a position to prove Theorem \ref{thm:STN=TN}.

We shall appeal to Theorem \ref{thm:Lorentzian21}, in particular the equivalence:
$$\phi^i_d(F) \ \text{is strongly totally nonnegative} \ \Leftrightarrow \ A_F \ \text{satisfies mixed HRR$_i$ on $U$}.$$

For convenience of the reader, we state Theorem \ref{thm:STN=TN} again:
\begin{theorem}
	\label{thm:TNNSTNN}
	If $\phi^i_d(F)$ is totally nonnegative then $A_F$ satisfies mixed HRR$_i$ on $U$.  In particular, $\phi^i_d(F)$ is totally nonnegative if and only if $\phi^i_d(F)$ is strongly totally nonnegative.
\end{theorem}
\begin{proof}
	The second statement follows from the first statement and Theorem \ref{thm:Lorentzian21}.
	To prove the first statement, we use induction on $d=\deg(F)$, the base case being clear.  Assume that $\phi^i_d(F)$ is totally nonnegative.  Then so are the submatrices $\frac{1}{d}\phi^i_{d-1}\left(x\circ F\right)$ and $\frac{1}{d}\phi^i_{d-1}\left(y\circ F\right)$.  By induction, it follows that the algebras $A_{x\circ F}$ and $A_{y\circ F}$ both satisfy mixed HRR$_i$ on $U$.  We claim that this implies that the algebra $A_F$ must satisfy mixed SL$_i$ on $U$.  Indeed, fix a sequence of linear forms in $U$, say $\mathcal{L}=(\ell_0,\ell_1,\ldots,\ell_d)\in U^{d+1}$, and suppose that for some $j$ satisfying $0\leq j\leq \min\{i,s(F)-1\}$, there exists $\alpha\in R_j$ satisfying 
	\begin{equation}
		\label{eq:Lj}
		\ell_1\cdots\ell_{d-2j}\cdot\alpha\circ F\equiv 0.
	\end{equation}
	We shall abuse notation slightly and let $\alpha$ also denote its image in the quotient ring $A_F$, or $A_{x\circ F}$ or $A_{y\circ F}$, when it is clear from the context what we mean.  
	Then multiplying both sides of Equation \eqref{eq:Lj} by either $x$ or $y$ then yields 
	$$\ell_1\cdots\ell_{d-2j}\cdot\alpha\circ\left(x\circ F\right)=0=\ell_1\cdots\ell_{d-2j}\cdot \alpha\circ\left(y\circ F\right).$$
	On the other hand, taking $$\mathcal{L}'=(\ell_0'=\ell_1,\ldots,\ell'_{d-1-2j}=\ell_{d-2j},\ell_{d-2j}'=\ell_{d-2j+1},\ldots,\ell_{d-1}'=\ell_d)$$
	we see that $\alpha$ must be primitive in the algebras $A_{x\circ F}$ and $A_{y\circ F}$ with respect to $\mathcal{L}'\in U^{(d-1)+1}$.  Since these algebras satisfy HRR$_i$ on $U$, it follows that if $\alpha\neq 0$ in $A_{x\circ F}$ then
	$$(-1)^j\cdot \ell_2\cdots\ell_{d-2j}\cdot\alpha\circ\left(x\circ F\right)=(-1)^j\cdot \ell_1'\cdots\ell_{d-1-2j}\cdot\alpha\circ\left(x\circ F\right)>0$$
	and similarly, if $\alpha\neq 0$ in $A_{y\circ F}$.  Therefore, writing $\ell_1=ax+by$ for some $a,b>0$, we have 
	\begin{align*}
		0= & (-1)^j\cdot \left(\ell_1\cdots\ell_{d-2j}\cdot\alpha\circ F\right)\\
		= &  (-1)^j\left(a\cdot \ell_2\cdots\ell_{d-2j}\cdot\alpha\circ \left(x\circ F\right)\right)\\
		& +(-1)^j\left(b\cdot \ell_2\cdots\ell_{d-2j}\cdot\alpha\circ\left(y\circ F\right)\right).
	\end{align*}
	In particular, $\alpha$ must be zero in both algebras $A_{x\circ F}$ and $A_{y\circ F}$.  This implies that $x\cdot \alpha=0$ and $y\cdot\alpha=0$ are both zero in the algebra $A_F$ and hence $\alpha\in \left(A_F\right)_j$ must either be zero in $A_F$ or it must be in the socle.  Since $j<d$ by our assumption, it follows that $\alpha$ must be zero in $A_F$, and hence the map $\times\ell_1\cdots\ell_{d-2j}\colon \left(A_F\right)_j\rightarrow\left(A_F\right)_{d-j}$ must be injective.  Since $\mathcal{L}\in U^{d+1}$ and $j$, $0\leq j\leq \min\{i,s(F)-1\}$ were chosen arbitrarily, it follows that $A_F$ must satisfy mixed SL$_i$ on $U$, as claimed.  
	
	It follows that for each $j$, $0\leq j\leq \min\{i,s(F)-1\}$, the permuted $j^{th}$ mixed Hessian determinant must be nonzero everywhere on the positive cone, i.e.
	$$\det\left(P_j\cdot \operatorname{MHess}_j(F,\mathcal{E})|_{(\underline{X},\underline{Y})}\right)\neq 0,$$
	for all $\left(\underline{X},\underline{Y}\right)>0$.
	Therefore, in order to conclude the induction, it will suffice to find, for each $0\leq j\leq \min\{i,s(F)-1\}$, \emph{some} point $\left(\underline{X},\underline{Y}\right)=\left(X_1,\ldots,X_{d-2j},Y_1,\ldots,Y_{d-2j}\right)>0$
	in the positive cone at which the permuted $j^{th}$ mixed Hessian determinant is positive.
	
	To find such a point, fix $j$ satisfying $0\leq j\leq \min\{i,s(F)-1\}$, and assume for the moment that $j<i$.  Then in this case we can specialize the variables $Y_1,\ldots,Y_{i-j}=t$ and the remaining variables to $X_1=\cdots=X_{d-2j}=Y_{i-j+1}=\cdots=Y_{d-2j}=1$, as in Lemma \ref{lem:Specialization}; in particular note that for any fixed $t>0$, this defines a point in the positive cone.  Then according to Lemma \ref{lem:Plucker} and Lemma \ref{lem:Specialization}, we can write the specialized permuted $j^{th}$ mixed Hessian determinant as the univariate polynomial
	$$M_j^F(t)=\sum_{K\in\binom{[d-j]_0}{j+1}}\Delta_K\left(\phi^j_d(F)\right)\cdot \Delta^i_{(K+j)J}(t)=\sum_{K\in\binom{[d-j]_0}{j+1}}\Delta_K\left(\phi^j_d(F)\right)\cdot t^{\alpha^i(K)}\cdot \hat{\Delta}^i_{(K+j)J}(t)$$
	where $\hat{\Delta}^i_{(K+j)J}(0)\neq 0$.
	Let $K=\{0\leq k_0<\cdots<k_j\leq d-j\}$ be any subset where $\Delta_K(\phi^j_d(F))$ is nonzero and $\alpha^i(K)$ is minimal.  If $i\leq k_j\leq d-j$, then according to Lemma \ref{lem:alpha}, there exists row and column index sets $I\subset\{0,\ldots,i\}$ and $J\subset\{0,\ldots,d-i\}$ such that  
	$$\Delta_{IJ}\left(\phi^i_d(F)\right)=\Delta_K\left(\phi^j_d(F)\right);$$
	of course in general, other terms may appear, but, according to Lemma \ref{lem:alpha} and our minimality assumptions on $\alpha^i(K)$, those other terms must vanish.
	Since $\phi^i_d(F)$ is totally nonnegative, it follows that $\Delta_K(\phi^j_d(F))$ must be strictly positive.  On the other hand, if $j\leq k_j\leq i-1$, then we can write
	$$\Delta_K(\phi^j_d(F))=\frac{1}{d^{j+1}}\cdot\Delta_K\left(\phi^j_{d-1}\left(y\circ F\right)\right).$$
	Since $A_{y\circ F}$ satisfies mixed HRR$_i$ on $U$, it follows from Theorem \ref{thm:Lorentzian21} that $\phi^i_{d-1}(y\circ F)$ is strongly totally nonnegative and hence $\phi^j_{d-1}\left(y\circ F\right)$ is totally nonnegative, and hence again, we must have that $\Delta_K\left(\phi^j_d(F)\right)>0$ is strictly positive.
	This shows that if $j<i$ and $0\leq j\leq \min\{i,s(F)-1\}$, then the univariate polynomial
	$M^F_j(t)$ has a positive coefficient on its lowest degree term, and hence $M^F_j(t)>0$ for $t>0$ sufficiently small.  If $j=i$, then we must have $\min\{i,s(F)-1\}=i$, and $\phi^i_d(F)$ is a full rank Toeplitz matrix, by \cite[Lemma 4.5]{MMS}.  It then follows directly from the Pl\"ucker expansion formula in Lemma \ref{lem:Plucker} that, since we are assuming $\phi^i_d(F)$ is totally nonnegative and of full rank, the permuted $i^{th}$ mixed Hessian determinant must be strictly positive on the positive cone, i.e. 
	$$\det\left(P_i\cdot\operatorname{MHess}_i(F,\mathcal{E})|_{(\underline{X},\underline{Y})}\right)=\sum_{K\in\binom{[d-i]_0}{i+1}}\Delta_K(\phi^i_d(F))\cdot\left(\sum_{\mathcal{P}\colon \mathcal{A}_{K+i}\rightarrow\mathcal{B}_J}\prod_{z=1}^{d-2i}X_z^{a_z(\mathcal{P})}\cdot Y_z^{b_z(\mathcal{P})}\right)>0,$$
	for all $(\underline{X},\underline{Y})>0$.
	In all cases, we must therefore conclude that permuted mixed Hessian polynomials are strictly positive on the positive cone up through degree $i$, and hence $A_F$ must therefore satisfy the mixed HRR$_i$ on $U$ by \cite[Lemma 3.13]{MMS}, completing the induction and the proof. 
\end{proof}

\begin{remark}
	\label{rem:Whitney2}
	\begin{enumerate}
		%	\item It is interesting to compare Theorem \ref{thm:Cor} with A. Whitney's theorem (\cite[Reference 42]{MMS}), which says 
		%$$\overline{\mathcal{M}(m,n)^{>0}}=\mathcal{M}(m,n)^{\geq 0}$$
		%or even the more recent theorem for Hankel matrices, proved by S. Fallat, C. Johnson, and A. Sokal (\cite[References 13,14]{MMS}), which says 
		%$$\overline{\mathcal{H}(m,n)^{>0}}=\mathcal{H}(m,n)^{\geq 0}.$$
		
		%It remains an open question whether or not the set of strongly totally nonnegative Toeplitz matrices equals the set of totally nonnegative Toeplitz matrices, i.e.
		%$$\overline{\mathcal{T}(m,n)^{>0}}=\mathcal{T}(m,n)^{\geq\geq 0}\overset{?}{=}\mathcal{T}(m,n)^{\geq 0}.$$
		%The equality in question holds for $m\leq 3$.
		
		%	Evidently, however, the Toeplitz matrices behave differently, as their totally nonnegative matrices can have isolated points, as in Example \ref{ex:Not}.
		
		\item 
		Theorem \ref{thm:Lorentzian2} together with Theorem \ref{thm:STN=TN} implies that for each $i$, $0\leq i< \flo{d}$, there is a descending inclusion 
		$$L(i+1)_d\subset L(i)_d;$$
		indeed if $\phi^{i+1}_d(F)$ is totally nonnegative then it must be strongly totally nonnegative and hence $\phi^i_d(F)$ is totally nonnegative too.  In particular, $L\left(\flo{d}\right)_d$ consists of polynomials that are $i$-Lorentzian for all $0\leq i\leq \flo{d}$, i.e.
		$$L\left(\flo{d}\right)_d=\bigcap_{0\leq i\leq \flo{d}}L(i)_d.$$ 
		One subfamily of $\flo{d}$-Lorentzian polynomials are those $d$-forms $F=\sum_{k=0}^d\binom{d}{k}c_kX^kY^{d-k}$ that are called normally stable \cite[Definition 4.25]{MMS}, characterized as those having totally nonnegative bi-infinite Toeplitz matrices $\phi_d(F)=\left(c_{q-p}\right)_{0\leq p,q}$, or alternatively, as those having dehomogenized normalizations $\tilde{F}(1,t)=\sum_{k=0}^dc_kt^{d-k}$ that have only real non-positive roots; this is \cite[Theorem 4 or Theorem 4.30]{MMS}.
		
		\item If $F\in L(s-1)_d$ where $s=s(F)$ is the Sperner number of $F$, then $F\in L\left(\flo{d}\right)_d$.  To see this, note
		that since the algebra $A_F$ satisfies mixed HRR$_{s-1}$ on $U$, it must therefore satisfy mixed HRR$_{\flo{d}}$ on $U$ since the primitive subspaces vanish in degrees $s\leq i\leq \flo{d}$.  It follows from Theorem \ref{thm:Lorentzian21} that $F\in L\left(\flo{d}\right)_d$ as well.  In terms of Toeplitz matrices, this says that if $\phi^{s-1}_d(F)$ is totally nonnegative, then so is $\phi^{\flo{d}}_d(F)$; this is the totally nonnegative analogue of Lemma \ref{lem:TPs}.  
	\end{enumerate}       
\end{remark}

\section{Other Errata}
\label{sec:Other}
Here are some other errata from \cite{MMS}.
	\begin{enumerate}
		\item page 21, proof of Corollary 4.13: the word ``precise'' should be replaced by the word ''explicit''
		\item page 25, the proof of Lemma 4.19:  there should be a $\frac{1}{d!}$ on the last line of the displayed equation.
		\item page 25-26, the proof of Theorem 4.21 contains the following false statements:
		\begin{enumerate}
			\item on page 25, the statement, after the last displayed equation ``...the minors of $\phi_d(F)_{IK}$ for $I=\{0,...,j\}$ and $K\subset \{j,...,d\}$ are precisely the consective maximal minors of $\phi^j_d(F)$, or the consecutive initial minors of $\phi^i_d(F)$ of size $j+1$...'' is false.
			\item on page 26, the statement after the third displayed equation
			``Since the minors of the $(j+1)\times (j+1)$ matrix $\phi^j_d(H_{t,u})$ are also minors of the matrix $\phi^i_d(H_{t,u})$ for all $j$,with $0\leq j\leq i$...'' is false.
		\end{enumerate}
		These statements have been corrected in our proofs of Theorem \ref{thm:Lorentzian21} and Theorem \ref{thm:TNNSTNN}.
		%\item Every instance of totally nonnegative Toeplitz matrix in connection to limits of totally positive Toeplitz matrices, or $i$-Lorentzian polynomials, or mixed HRR$_i$ on $U$, should be replaced by the term strongly totally nonnegative defined above.  These include 
		%\begin{itemize}
		%	\item page 1, abstract, lines 4, 7
		%	\item page 2, introduction, lines -2, -7
		%	\item page 3, Theorem 2, item 2
		%	\item page 3, Theorem 3, line 2
		%	\item page 3, line -7
		%	\item page 22, line -7
		%	\item page 28, after Remark 4.29, line 1
		%\end{itemize}  
	\end{enumerate}

\bigskip

%\bibliographystyle{amsalpha}
%\bibliographystyle{plain}
%\bibliography{Lorentzianreferences.bib}

\end{document}